\documentclass[12pt]{article}

\usepackage{amsfonts}
\usepackage{amssymb}
\usepackage{amsthm}
\usepackage{amsmath}
\usepackage{amscd}
\usepackage{mathrsfs}
\usepackage{enumerate}
\usepackage{ulem}
\usepackage{fontenc}
\usepackage[all]{xy}
\usepackage{array}
\usepackage[english]{babel}

\def\CC {{\mathbb F}}     
\def\FF {{\mathbb F}}     
\def\NN {{\mathbb N}}     
\def\PP {{\mathbb P}}     
\def\XX {{\mathbb X}}     

\def\mc {\mathcal}
\def\mk {\mathfrak}

\def\ol  {\overline}

\def\rw  {\rightarrow}

\def\ul  {\underline}

\DeclareMathOperator{\Hom}{\mathrm{Hom}}\DeclareMathOperator{\Ima}{\mathrm{Im}}\DeclareMathOperator{\Sp}{\it{Sp}}\DeclareMathOperator{\T}{{\it T}}\DeclareMathOperator{\SO}{\mathrm{SO}}\DeclareMathOperator{\codim}{\mathrm{codim}}\DeclareMathOperator{\Supp}{\mathrm{Supp}}\DeclareMathOperator{\rk}{\mathrm{rk}}
\newcounter{tAP}

\newcommand{\APbr}[1]{\langle#1\rangle}

\theoremstyle{definition}
\newtheorem{theorem}{Theorem}[section]
\newtheorem{lemma}[theorem]{Lemma}

\newtheorem{prop}[theorem]{Proposition}

\newtheorem{coro}[theorem]{Corollary}
\newtheorem{rem}{Remark}[section]
\newtheorem{df}{Definition}[section]

\newcounter{AP}

\begin{document}

\begin{center}

\large{\bf HOMOGENEOUS PROJECTIVE VARIETIES}\\

\large{\bf WITH SEMI-CONTINUOUS RANK FUNCTION}\\

\vspace{0.5cm}

\large{\textsc{A. V. Petukhov \footnote{Institute for Information Transmission Problems, Bol'shoy Karetniy 19, Moscow, Russia.\\ Email: alex-~\!\!-~2@yandex.ru.} and V. V. Tsanov \footnote{Fakult\"at f\"ur Mathematik, Ruhr-Universit\"at Bochum, Bochum, Germany.\\ Email: valdemar.tsanov@gmail.com.}}}\\

\vspace{0.5cm}

\begin{abstract}Let $\XX\subset\PP(V)$ be a projective variety, which is not contained in a hyperplane. Then
every vector $v$ in $V$ may be written as a sum of vectors from the
affine cone over $\XX$. The minimal number of summands in such a
sum is called the rank of $v$. In this paper, we classify
all equivariantly embedded homogeneous projective varieties
$\XX\subset\PP(V)$ whose rank function is lower semi-continuous. Classical examples
are: the variety of rank one matrices (Segre variety with two
factors) and the variety of rank one quadratic forms (quadratic
Veronese variety). In the general setting, $\XX$ is the orbit in
$\PP(V)$ of a highest weight line in an irreducible representation
$V$ of a reductive algebraic group $G$. Thus, our result is a list
of all irreducible representations of reductive groups, for which the corresponding rank function is lower semi-continuous.
\end{abstract}

\end{center}

\small{
\tableofcontents
}

\section{Introduction}\label{Sintro}

Let $V$ be a finite dimensional vector space over an algebraically closed field $\FF$ of characteristic $0$. Let $\XX\subset\PP(V)$ be a projective variety and $X\subset V$ be the affine cone over $\XX$. Suppose that $\XX$ is nondegenerate, i.e. it is not contained in a hyperplane. Then every vector $v\in V$ can be written as a linear combination of points of $X$ and we have a well-defined function ${\rm rk}:\PP(V)\to\NN$, called {\it rank} (or $\XX$-{\it rank}), given by
$$
{\rm rk}[v]={\rm rk}_\XX[v] = \min\{r\in\NN: v=x_1+...+x_r,\; {\rm with}\; x_j\in X\} \;,
$$
where $[v]\in\PP(V)$ denotes the projective point corresponding to a non-zero vector $v$. The {\it rank sets} in $\PP(V)$ with respect to $\XX$ are defined as
$$
\XX_r = \{[v]\in\PP(V): {\rm rk}[v]=r\} \;.
$$
The {\it secant varieties} of $\XX$ are defined as the Zariski closure
$$
\sigma_r(\XX) = \ol{\bigcup_{s\leq r}\XX_s} \;.
$$
The {\it border rank} of $[v]\in\PP(V)$ is defined as
$$
\ul{\rm rk}[v]=\ul{\rm rk}_\XX[v] = \min\{r\in\NN:[v]\in\sigma_r(\XX)\} \;.
$$
Let ${\rm Aut}(\XX)\subset SL(V)$ denote the group of linear automorphisms of $\XX$. Then rank and border rank are ${\rm Aut}(\XX)$-invariant, and hence ${\rm Aut}(\XX)$ acts on $\XX_r$ and $\sigma_r(\XX)$.

\begin{df}
We call $\XX$ {\it rs-continuous} if ${\rm rk}_\XX$ is a lower semi-continuous function on $\PP(V)$, i.e. if rank and border rank defined by $X$ coincide. Otherwise, we say that $\XX$ is {\it r-discontinuous}. We call $[v]\in\PP(V)$ {\it exceptional} if ${\rm rk}_\XX[v]\ne\ul{\rm rk}_\XX[v]$.
\end{df}

Our goal is to classify all rs-continuous varieties belonging to a certain class, namely, the homogeneous rs-continuous varieties. By a homogeneous projective variety we mean an equivariantly embedded variety $\XX\subset\PP(V)$ whose automorphism group is transitive. In such a case the (linear) automorphism group $G$ is always semisimple and $\XX\cong G/P$, where $P$ is a parabolic subgroup of $G$. In other words, $\XX$ is a flag variety of $G$. Furthermore, when the embedding is nondegenerate, $V$ is an irreducible $G$-module, and so it is determined up to isomorphism by its highest weight, say $\lambda$, so that $V\cong V(\lambda)$. Then $\XX$ can be viewed as the orbit of a highest weight line $\XX=G[v^\lambda]\subset\PP(V)$. Conversely, if $G$ is a semisimple algebraic group and $V=V(\lambda)$ is an irreducible $G$-module, then $\PP(V)$ contains a unique closed $G$-orbit --- the orbit of a highest weight line; we denote this orbit by $\XX(G,V)$. It is not always true that $G$ is the full automorphism 
group
of $\XX(G, V)$. For instance, the symplectic group $Sp_{2n}$ acts transitively on the projective space $\PP(\CC^{2n})$, but the full linear automorphism group is $SL_{2n}$. Our approach is to classify all irreducible representations $(G,V)$ such that $\XX(G,V)$ is rs-continuous. Then the list of rs-continuous homogeneous projective varieties is obtained by dropping the redundancies. For brevity of expression, we shall call a representation $(G,V)$ {\it rs-continuous}, if the corresponding variety $\XX(G,V)$ is rs-continuous.

Before stating our classification theorem, let us mention some classical examples, where the terminology stems from.

Let $V=\CC^m\otimes\CC^n$ be the space of $m\times n$-matrices. On $V$ we have the classical notion of rank of a matrix. Let $\XX\subset \PP(V)$ denote the variety of matrices of rank 1; $\XX$ can be defined by the vanishing of all $2\times 2$-minors and is also known as the Segre variety of simple tensors $\XX={\rm Segre}(\PP^{m-1}\times\PP^{n-1})$. It is well-known that every matrix of rank $r$ can be written as a sum of $r$ matrices of rank 1. The set of matrices of rank $r$ or less is the common zero-locus of all $(r+1)\times(r+1)$-minors and hence is a closed set, which equals the secant variety $\sigma_r(\XX)$. Hence $\XX$ is rs-continuous. The automorphism group of $\XX$ equals $PSL_m\times PSL_n$ in this case.

A similar situation occurs for symmetric and for skew symmetric matrices. The corresponding varieties are the quadratic Veronese embedding ${\rm Ver}_2(\PP^{n-1})\subset\PP(S^2\CC^n)$ and the Grassmannian of planes ${\rm Gr}_2(\CC^n)\subset\PP(\Lambda^2\CC^n)$. In both cases we have a linear action of $SL_n$ preserving $\XX$ and yielding the full automorphism group of $\XX$. Both varieties are rs-continuous.

Perhaps the simplest r-discontinuous homogeneous variety is the twisted cubic curve: $\XX={\rm Ver}_3(\PP^1)\subset\PP(S^3\CC^2)$. The respective automorphism group $G=PSL_2$ has three orbits in $\PP(S^3\CC^2)$, which can be written as $G[x^3]$, $G[x^2y]$ and $G[x^3+y^3]$, where $x,y$ is an arbitrary basis of $\CC^2$. Indeed, an element of $\PP(S^3\CC^2)$ can be written as a product $[l_1l_2l_3]$, with $[l_j]\in\PP^1$, and there are three possibilities: either $[l_1]=[l_2]=[l_3]$, or $[l_1]=[l_2]\ne [l_3]$, or all three are distinct. Since $PSL_2$ acts transitively on triples of distinct points on $\PP^1$, we have three orbits in $\PP(S^3\CC^2)$. It is easy to see that the ranks are $1$, $3$ and $2$, respectively. However, we have $\sigma_2(\XX)=\ol{G[x^3+y^3]}=\PP(S^3\CC^3)$ and hence $[x^2y]$ is exceptional and $\XX$ is r-discontinuous.

Let us notice that, in the above example, $\ol{G[x^2y]}$ is the tangential variety of $\XX$, which is clearly contained in the secant variety $\sigma_2(\XX)$. In fact it is a general phenomenon, that exceptional points do appear in the tangential variety, whenever they exist. This fact is in the basis of our methods.

Now, we formulate our main result.

\begin{theorem}\label{Tkltame} Let $G$ be a semisimple algebraic group and $V$ be a finite dimensional irreducible $G$-module, with $\dim V\ge2$. Then the closed $G$-orbit $\XX(G, V)\subset\PP(V)$ is rs-continuous if and only if the pair $(G, V)$ appears in the following table.

\begin{equation}\begin{tabular}{|c|}\hline$\begin{tabular}{c|c|c}Group
$G$&Representation $V$&Highest weight of $V$\\\end{tabular}$\\
\hline Simple classical groups\\\hline
$\begin{tabular}{c|c|c}$SL_n$&$\begin{tabular}{c}$\CC^n, (\CC^n)^*, (\Lambda^2\CC^n), (\Lambda^2\CC^n)^*,$\\ S$^2\CC^n$, (S$^2\CC^n)^*$, $\mk{sl}_n$ \end{tabular}$&$\begin{tabular}{c}$\pi_1$, $\pi_{n-1}$, $\pi_2$, $\pi_{n-2}$,\\ $2\pi_1$, $2\pi_{n-1}$, $\pi_1+\pi_{n-1}$\end{tabular}$\\\hline$SO_n$&$\CC^n, RSpin_n(n\le10)$&$\begin{tabular}{c}$\pi_1$, $\pi_{\frac n2} (2\mid n)$, $\pi_{\frac{n}{2}-1}(2\mid n)$,\\ $\pi_{\frac{n-1}2} (2\nmid n)$\end{tabular}$\\\hline$Sp_{2n}$&$\CC^{2n}, \Lambda^2_0\CC^{2n},$ S$^2\CC^{2n}\cong\frak{sp}_{2n}$&$\pi_1, \pi_2, 2\pi_1$\\\end{tabular}$\\
\hline Simple exceptional groups\\\hline
$\begin{tabular}{c|c|c}$E_6$&$\CC^{27}, (\CC^{27})^*$&$\pi_1,
\pi_5$\\\hline$F_4$&$\CC^{26}$&$\pi_1$\\\hline$G_2$&$\CC^7$&$\pi_1$\\\end{tabular}$\\
\hline Non-simple groups\\\hline $\begin{tabular}{c|c|c}$SL_m\times
SL_n$&$\CC^m\otimes\CC^n$&$\pi_1\oplus\pi_1$\\\hline$SL_m\times
Sp_{2n}$&$\CC^m\otimes\CC^{2n}$&$\pi_1\oplus\pi_1$\\\hline
$Sp_{2m}\times
Sp_{2n}$&$\CC^{2m}\otimes\CC^{2n}$&$\pi_1\oplus\pi_1$\\\end{tabular}$\\\hline
\end{tabular}\hspace{10pt},\label{T1}\end{equation}
where by $RSpin_n$ we denote (any) spinor representation of the
simply connected cover of $SO_n$, by $\Lambda^2_0\CC^{2n}$ we denote
the second fundamental representation of $Sp_{2n}$ (which is
identified with a hyperplane in $\Lambda^2\CC^{2n}$), by $\CC^{27}$
we denote any one of the two smallest fundamental representations of
$E_6$, by $\CC^{26}$ we denote the smallest fundamental
representation of $F_4$, by $\CC^7$ we denote the smallest
fundamental representation of $G_2$. In the third column we list
the highest weight of $V$ with respect to $G'$ (we use here and
throughout the paper the numbering convention
of~\cite[p. 294]{VO}).

Moreover, if $(G,V)$ is r-discontinuous, then it contains an exceptional vector of border rank 2.
\end{theorem}

We give a proof of Theorem~\ref{Tkltame} in Section~\ref{SprT}. As a corollary, we obtain the list of homogeneous projective varieties given below. The list of varieties is shorter,
because, as mentioned above, in certain cases there are subgroups of the automorphism group of the variety acting transitively.

\begin{coro}\label{Coro ListTameVar}
The rs-continuous projective varieties $\XX\subset\PP(V)$ with transitive linear automorphism group are the following:
\begin{center}
\begin{tabular}{|c|c|c|c|}
\hline
Notation for $\XX$ & Ambient~$\PP(V)$ & Group $G$& Max rk$_\XX$ \\
\hline
$\PP(\FF^n)$ &$\PP(\FF^n)$& $SL_n$ & $1$ \\
\hline
${\rm Ver}_2(\PP(\FF^n))$&$\PP({\rm S}^2\FF^n)$ & $SL_n$ & $n$ \\
\hline
Gr$_2(\FF^n)$&$\PP(\Lambda^2\FF^n)$ & $SL_n$ & $\lfloor\frac{n}{2}\rfloor$ \\
\hline
Fl$(1,n-1; \FF^n)$&$\PP(\mk{sl}_n)$ & $SL_n$ & $n$ \\
\hline
Q$^{n-2}$&$\PP(\FF^n)$ & $SO_n$ & $2$ \\
\hline
S$^{10}$&$\PP(\FF^{16})$ & $Spin_{10}$ & $2$ \\
\hline
Gr$_{\omega}(2,\FF^{2n})$&$\PP(\Lambda_0^2\FF^{2n})$ & $Sp_{2n}$ & n \\
\hline
E$^{16}$&$\PP(\FF^{27})$ & $E_6$ & $3$ \\
\hline
F$^{15}$&$\PP(\FF^{26})$ & $F_4$ & $3$ \\
\hline
${\rm Segre}(\PP(\FF^m)\times\PP(\FF^n))$&$\PP(\FF^m\otimes\FF^n)$ & $SL_m\times SL_n$ & $\min\{m,n\}$ \\
\hline
\end{tabular}
\end{center}
\end{coro}

It is natural to ask whether rs-continuous varieties admit another general characterization. In fact, the starting point of our study was a result by Buczy\'nski and Landsberg, \cite{Bucz-Lands-2011}, based on previous work by Landsberg and Manivel, \cite{Lands-Mani-2003}. This result exhibits a remarkable class of rs-continuous homogeneous varieties --- the subcominuscule varieties. Recall that a variety $\XX\subset\PP(V)$ is called subcominuscule if it is the variety associated to the isotropy representation of an irreducible Hermitian symmetric space $S$, i.e. $\XX=\XX(G,V)$, with $G$ being the complexification of the semisimple part of isotropy subgroup the isometry group of $S$, and $V$ being the tangent space. Then \cite[Prop. 4.1]{Bucz-Lands-2011} states that ${\rm rk}_\XX=\ul{\rm rk}_\XX$ and, furthermore, the $G$-orbits in $\PP(V)$ are exactly the rank sets $\XX_r$. It is then natural to ask: are there other rs-continuous homogeneous varieties besides the subcominuscule ones? There are. It was shown
by Kaji and Yasukura, \cite{Kaji-Yasukura-2000}, that the adjoint variety $\XX(G,{\mk g})$ of a simple Lie algebra ${\mk g}$ is rs-continuous if and only if ${\mk g}$ is of type $A_n$ or $C_n$; see also \cite{Kaji-SecAdjVar}, \cite{Baur-Draisma-2004}. The adjoint variety of type $C_n$ is just the quadratic Veronese variety, which is subcominuscule with respect to its automorphism group. However, the adjoint variety of type $A_n$ is not subcominuscule. The non-subcominuscule varieties in our list are Fl$(1,n-1;\FF^n)$, Gr$_\omega(2,\FF^{2n})$ and F$^{15}$. All these exceptions are, however, hyperplane sections in subcominuscule varieties. Conversely, the homogeneous hyperplane section of subminuscule varieties are the above three and the quadric $Q^{n-1}$ viewed as a hyperplane section in ${\rm Ver}_2(\PP^n)$. This follows from a classification of homogeneous hyperplane sections in homogeneous projective varieties given by Watanabe \cite{Wata-2011}. Thus our theorem implies that, if $\XX$ is a homogeneous
hyperplane section in a subcominuscule variety $\tilde\XX$ in its minimal projective embedding, then $\XX$ is rs-continuous. So we can formulate the following:

\begin{coro}
A homogeneous projective variety $\XX\subset\PP(V)$ is rs-continuous if and only if it is either a subcominuscule variety, or a hyperplane section in a subcominuscule variety $\tilde\XX\subset\PP(\tilde V)$ in its minimal projective embedding. In the latter case, the rank function of $\XX$ is equal to the restriction of the rank function of $\tilde \XX$ and the secant varieties of $\XX$ are equal to the intersections of the secant varieties of $\tilde \XX$ with $\PP(V)$.
\end{coro}

Since our approach is to start with a representation $(G,V)$ rather than with a homogeneous variety $\XX$, we have found the following observations useful. If $(G,V)$ is a representation such that $G$ acts spherically on the projective space $\PP(V)$, then the variety $\XX(G,V)$ is subcominuscule, i.e. $(Aut(\XX),V)$ is a subminuscule representation. This follows directly from the classification of spherical representations given by Kac, \cite{Kac-1980}, see also \cite{Knop-MultFree-1998}.

\begin{rem}
Since spherical representations have finitely many orbits and, by the above corollary, rs-continuous representations are not far away from spherical, it makes sense to ask whether this class of representations is related to the well-known class of projective representations with finitely many orbits. The fact that twisted cubic is not rs-continuous together with the fact that $PSL_n (n\ge3)$ has infinitely many orbits on $\PP(\frak{sl}_n)$ show that the class of rs-continuous representations neither contains, nor is contained in, the class of projective representations with finitely many orbits. Let us notice, however, that the automorphism group of an rs-continuous homogeneous projective variety has finitely many orbits in the projective space if and only if the variety is subcominuscule.
\end{rem}


One of the most important questions, which is asked in the studies of secant varieties and rank is: what are the ideals of secant varieties? It is known, by a result of Kostant, that the ideal of $\XX$ is generated in degree $2$, by the appropriate generalization of the Pl\"ucker equations; cf. \cite[Th. 16.2.2.6]{Landsberg-2012-book}. It was shown by Landsberg and Manivel, \cite{Lands-Mani-2003}, that, for a subcominuscule variety $\XX\subset\PP(V)$, the ideal of the $r$-th secant variety $\sigma_r(\XX)$ is generated in degree $r+1$ by the $(r-1)$-th prolongation of the generating set of the ideal of $\XX$, which is defined as $I_2(\XX)^{(r-1)}=(I_2(\XX)\otimes S^{r-1}V^*)\cap S^{r+1}V^*$. Our theorem has the following

\begin{coro}
If $\XX\subset\PP(V)$ is an rs-continuous homogeneous variety, then the ideal of $\sigma_r(\XX)$ is generated in degree $r+1$ by the prolongation $I_2(\XX)^{(r-1)}$.
\end{coro}

Let us emphasize, that the second secant variety $\sigma_2(\XX)$ plays a prominent role
both in this paper and in the literature. Sometimes the second
secant variety is called just ``the secant variety''. This variety
is much more accessible than the higher secant varieties: for
example for a simple $G$-module the second secant variety has an
open $G$-orbit~\cite[Ch. III, Thm 1.4]{Zak-Book}. Furthermore, we have $\sigma_2(\XX)=\XX_2\cup T\XX$, where $T\XX$ is the tangential variety of $\XX$, cf. \cite[\S 8.1]{Landsberg-2012-book}. It turns out that, if exceptional points exist, they are always present in $T\XX$.

For a systematic treatment, as well as an extensive bibliography, on secant varieties and rank we refer the reader to the recent book of Landsberg, \cite{Landsberg-2012-book}. The general theory of secant varieties allows one to deduce rs-continuity for varieties of small codimension, see Corollary \ref{Cz} here. However, this criterion is applicable to relatively few homogeneous varieties, and to none of the more difficult cases.

The paper is organized as follows. In Subsection~\ref{Sec Rank&Sec},
we recall the notions of secant varieties, rank and border rank,
with their basic properties. In Subsection~\ref{SSirep}, we recall
some basic notions about algebraic groups: Borel subgroup, Cartan
subgroup, weight lattice, root system, Weyl chamber. We also
introduce the notion of chopping (this is a simple combinatorial
procedure) and provide some facts on $\XX_2(G, V)$ and
$\sigma_2(\XX(G, V))$ playing a crucial role in this paper.

In Section~\ref{SprT}, we present a plan of our proof of
Theorem~\ref{Tkltame}, the main theorem of our article. Essentially,
this proof is a compilation of Propositions~\ref{Lh3},~\ref{Llist}
and Theorems~\ref{Pftamecl} and~\ref{Pftameex}. We prove
Propositions~\ref{Lh3},~\ref{Llist} in Section~\ref{Sgm}. We prove
Theorems~\ref{Pftamecl},~\ref{Pftameex} in Section~\ref{Sfrepcl}
and~\ref{Sfrepex} respectively.

In Section~\ref{SSrls}, we prove a strong necessary condition for rs-continuity
of a representation in terms of its choppings. This is formulated in Proposition~\ref{Predl}.

The main statement of Section~\ref{Sfrepcl} is
Theorem~\ref{Pftamecl}. In this theorem we find out which
fundamental representations of classical groups are rs-continuous and which
are r-discontinuous. This is done in the following way: for the fundamental modules
\begin{center}$\FF^n, \Lambda^2\FF^n$ for $SL_n$, $\Lambda^3\FF^6$ for $SL_6$, $\Lambda^2_0\FF^{2n}$ and $\Lambda^3_0\FF^{2n}$ for $Sp_{2n}$,\\$\Lambda^2\FF^n$ for $SO_n$, $RSpin_n$ for $Spin_n (n\le 12)$\end{center}
we check rs-continuity/r-discontinuity in a straightforward way. From these data we deduce r-discontinuity of all other modules using the notion of chopping.

The main statement of Section~\ref{Sfrepex} is
Theorem~\ref{Pftameex}. In this theorem we find out which
fundamental representations of exceptional groups are rs-continuous and which
are r-discontinuous. This is done via case by case checking of the 27 fundamental
representations of the 5 exceptional Lie algebras. For any such
representation we find some arguments by which it is r-discontinuous/rs-continuous. For
most of the representations the arguments are quite short, using chopping or a reference, but for
three representations:\begin{center}$V(\pi_1), V(\pi_2)$ for $F_4$
and $V(\pi_1)$ for $E_7$\end{center}we are able to find only relatively
long arguments presented in the corresponding subsections.

\section{Preliminaries}

In this section, we recall some definitions and elementary facts
about secant varieties and rank. The goal is to introduce notation
and perhaps help the unexperienced reader to become more familiar
with these notions. We also fix some standard notation for reductive
algebraic groups, their Lie algebras and their representations.

Throughout the paper we use the following notation.  The letter
$\XX$ is always used for a projective variety and $X$ denotes the
affine cone over it. For any subset $S\subset V$ we denote by
$\APbr{S}\subset V$ the span of $S$. For any subset $\mathbb
S\in\PP(V)$ we denote by $\APbr{S}\subset V$ the span of the cone
$S$ of $\mathbb S$ in $V$. For any non-zero vector $v\in V$ we
denote by $[v]$ the class of it in $\PP(V)$. If $v=0$, we set
$[v]:=0$.

\subsection{Secant varieties and rank: general definitions}\label{Sec Rank&Sec}

Let $\XX\subset\PP$ be an algebraic variety and $X \subset V$ denote
the affine cone over $\XX$. We denote by $\PP\APbr{\XX}=\PP(\APbr
X)$ the corresponding projective subspace of $\PP$. We say that
$\XX$ spans $\PP$ if $\PP\APbr\XX=\PP$; this is equivalent to the
requirement that $X $ contains a basis of $V$. Assume that this is
the case. Then every point in ${V}$ can be written as a linear
combination of points in $ X $. This allows us to define the notion
of rank already given in the introduction: the rank of
$[\psi]\in\PP$ with respect to $\XX$ is the minimal number of
elements of $X$ necessary to express $\psi$ as a linear combination.
Thus, the space $\PP$ is partitioned into the rank subsets,
$$\PP=\XX_1\sqcup\XX_2\sqcup ....$$ Since $\XX$ spans $\PP$, we have
$\XX_r=\emptyset$ for $r>\dim\PP$.

The following properties of varieties $\XX_r$ hold:

{\rm (i)} $\XX_1=\XX$.

{\rm (ii)} There exists a maximal $r_m\in\{1,...,\dim V\}$, such that $\XX_{r_m}\ne \emptyset$ and $\XX_{r}=\emptyset$ for $r>r_m$.

{\rm (iii)} If $r\in\{1,...,r_m\}$, then $\XX_{r}\ne \emptyset$.

{\rm (iv)} The projective space $\PP$ can be written as a disjoint
union $\PP=\XX_1\sqcup \dots \sqcup \XX_{r_m}$.

Let $r\in\{2,...,r_m\}$. The subset $\XX_r\subset\PP$ is not closed, because we have $\XX\subset\ol{\XX_r}$ and $\XX\nsubseteq\XX_r$. (Here and in what follows we use $\ol{S}$ to denote the Zariski closure of a subset $S$ of some algebraic variety.) The $r$-th secant variety of $\XX$ is defined as
$$
\sigma_r(\XX) = \ol{\bigsqcup\limits_{s\leq r}\XX_s} \subset\PP \;.
$$
It can also be written as
$$
\sigma_r(\XX) = \ol{\bigcup\limits_{x_1,...,x_r\in\XX} \PP_{x_1...x_r} } \;,
$$
where $\PP_{x_1...x_r}$ stands for the projective subspace of $\PP$ spanned by the points $x_1,...,x_r$.

The following properties of secant varieties
$\sigma_r(\XX)$ hold:

{\rm (i)} $\sigma_1(\XX)=\XX_1=\XX$.

{\rm (ii)} $\sigma_r(\XX)\subset\sigma_{r+1}(\XX)$.

{\rm (iii)} If $\XX$ is irreducible, then $\sigma_r(\XX)$ is also irreducible.

{\rm (iv)} There exists a minimal number $r_g\in\{1,...,r_m\}$ such that $\sigma_{r_g}(\XX)=\PP$ and $\sigma_{r_g-1}(\XX)\ne\PP$.

{\rm (v)} For $r\in\{1,...,r_g\}$ the rank subset $\XX_r$ is dense in $\sigma_r(\XX)$, i.e. we have $\sigma_r(\XX)=\ol{\XX_r}$.

\begin{df}
The number $r_g$ from part {\rm (iv)} of the above proposition is called the {\it typical rank} of $\PP$ with respect to $\XX$.
\end{df}

Let $[\psi]\in\PP$. The border rank of $[\psi]$ with respect to $\XX$ is defined as
$$
\ul{\rm rk}[\psi]:= \ul{\rm rk}_\XX[\psi]:=
\min\{r\in\NN:[\psi]\in\ol{\XX_r}\} \;.
$$

\begin{df}
Points $[\psi]\in\PP$, for which ${\rm rk}[\psi]\ne \ul{\rm rk}[\psi]$, are called {\it exceptional}.
\end{df}

Clearly, ${\rm rk}[\psi]\geq\ul{\rm rk}[\psi]$ and $[\psi]$ is
exceptional exactly when $\ul{\rm rk}[\psi]<{\rm rk}[\psi]$. So,
exceptional points are points which can be approximated by points of
lower rank. Also, we have $$\ul{\rm rk}[\psi] =
\min\{r\in\NN:[\psi]\in\sigma_r(\XX)\}.$$ This leads us to the next
definition.

\begin{df} (i) The secant variety $\sigma_r(\XX)$ is called {\it r-discontinuous} if it
contains an exceptional vector and {\it rs-continuous} if it does not.

(ii) The embedding $\XX\subset\PP$ is called {\it
rs-continuous}, if all secant varieties $\sigma_r(\XX)$ are
rs-continuous. Equivalently, $\XX\subset\PP$ is rs-continuous if ${\rm rk}$ is a lower semi-continuous function on $\PP$. We say that $\XX\subset\PP$ is {\it r-discontinuous} if $\XX\subset\PP$ is not rs-continuous.

(iii) The embedding $\XX\subset\PP$ is called {\it $i$-continuous}, if
$\sigma_i(\XX)$ is rs-continuous. The embedding $\XX\subset\PP$ is called
$i$-{\it discontinuous}, if $\sigma_i(\XX)$ is r-discontinuous and $\sigma_s(\XX)$ is rs-continuous for $s<i$.
\end{df}

We record another list of simple statements, which are derived
immediately from the above definitions.

1) The secant variety $\sigma_r(\XX)$ is r-discontinuous if and only if $\sigma_r(\XX)\ne \XX_1\sqcup\XX_2\sqcup...\sqcup\XX_r$.

2) The embedding $\XX\subset \PP$ is r-discontinuous if and only if ${\rm rk}_\XX\ne\ul{\rm rk}_\XX$.

We denote by $X_r\subset V$ the cone over $\XX_r$ without 0 and by
$\sigma_r(X)\subset V$ the cone over $\sigma_r(\XX)$ with 0. We denote by $T\XX$ the union of over all points of $\XX$ of tangent spaces to $X$ in $\PP$. We say that $\sigma_2(\XX)$ is {\it nondegenerate} if $\dim\sigma_2(\XX)=2\dim\XX+1$. We say that $T\XX$ is {\it nondegenerate} if $\dim T\XX=2\dim\XX$. The following propositions provide to us abstract tools which we use to study rs-continuity/r-discontinuity of varieties $\XX\subset\PP$.
\begin{prop}[{~\cite[Corollary 4]{FH}}] If $\XX$ is a smooth projective variety, then precisely one of the following must hold:

(i) $\dim T\XX=2\dim\XX$ and $\dim\sigma_2(\XX)=2\dim\XX+1$, or

(ii) $T\XX=\sigma_2(\XX)$.\end{prop}
\begin{prop}[{Zak's theorem on linear normality ~\cite{FL},~\cite{Zak-Book}, see also~\cite[Theorem 1.1]{L96}}] Assume that $\XX$ is smooth, nondegenerate and $\PP\ne\sigma_2(\XX)$. Then $\codim_\PP \XX\ge \frac{\dim\XX}2+2$.\end{prop}
\begin{coro}\label{Cz} Assume that $\XX$ is smooth, nondegenerate and $\codim_\PP \XX< \frac{\dim\XX}2+2$. Then $\XX$ is rs-continuous.\end{coro}

\begin{prop}[{Landsberg-Roberts~\cite[Theorem 10.3]{L96},~\cite[Introduction]{Rob}}]\label{Pl96} Assume that $\XX$ is smooth and $\codim_\PP\XX>\binom{\dim X+1}2$. Then $\sigma_2(\XX)$ is nondegenerate.\end{prop}



Some representations satisfy the conditions of Corollary~\ref{Cz} and thus their rs-continuity is checked by the general machinery. Nevertheless, to study all possible representations we need more methods. The general methods we employ to decide rs-continuity for nonfundamental representations are independent of codimension of the variety or degeneracy of its secant variety.

\subsection{Irreducible representations of reductive groups}\label{SSirep}

Here we fix some notation concerning semisimple or reductive
algebraic groups and their representations. All notions from this
theory used by us can be found in \cite{Goodman-Wallach-2009} (the
notation, however, may differ).

Let $G$ be a connected reductive algebraic group over $\CC$ and
${\mk g}$ be its Lie algebra. We assume that the semisimple part
$G'$ of $G$ is simply connected. Let $B\subset G$ be a Borel
subgroup and $T\subset B$ be a maximal torus. Let ${\mk h}\subset{\mk
b}\subset {\mk g}$ denote the respective subalgebras of ${\mk g}$.
Let $\Lambda\subset{\mk h}^*$ be the integral weight lattice and
$\Delta\subset\Lambda$ the root system. Let
$\Delta=\Delta^+\sqcup\Delta^-$ be the partition of the root system
into positive and negative roots corresponding to the Borel subgroup
$B$. Let $\Pi=\{\alpha_1,...,\alpha_\ell\}$ be the basis of simple
roots in $\Delta^+$.

The Cartan-Killing form of $\frak g$ determines a scalar product
$(\cdot, \cdot)$ on $\frak h^*$. This scalar product $(\cdot,
\cdot)$ defines a Dynkin diagram (graph) $Dyn$, whose vertices are
labeled by the simple roots $\alpha_1,...,\alpha_\ell\in\Pi$. Also
$\Pi$ and the scalar product define the dominant Weyl chamber
$\Lambda$ and the monoid of dominant weights
$\Lambda^+\subset\Lambda$, generated by the fundamental weights
$\pi_1,...,\pi_\ell$. Hence $\lambda\in\Lambda^+$ defines a function
$f_\lambda:\Pi\to\mathbb Z_{\ge0}$ such that
$\lambda=f_\lambda(\alpha_1)\pi_1+...+f_\lambda(\alpha_\ell)\pi_\ell$. This defines a one-to-one correspondence between the set $\Lambda^+$ of dominant weights and functions from the set of vertices of $Dyn$ to the non-negative integers. We put
\begin{center}$h(\lambda):=f_\lambda(\alpha_1)+...+f_\lambda(\alpha_\ell)$.\end{center}

The set $\Lambda^+$ is also in a one-to-one correspondence with the set of isomorphism classes of simple finite-dimensional
$G$-modules. We denote by $V(\lambda)$ the irreducible
representation of $G$ corresponding to $\lambda\in\Lambda^+$ ($\lambda$
is a highest weight of $V(\lambda)$ and $V(\lambda)$ contains a
unique up to scaling vector $v^\lambda$ of weight $\lambda$). Set
$\PP(\lambda):=\PP(V(\lambda))$. We denote by $X(\lambda)$ and
$\XX(\lambda)$ the orbits of $v^\lambda$ and $[v^\lambda]$ in
$V(\lambda)$ and $\PP(\lambda)$ respectively. Note that
$\XX(\lambda)$ is a unique closed $G$-orbit on $\PP(\lambda)$.

Any subdiagram of a Dynkin diagram is again a Dynkin diagram. Thus,
by chopping down some vertices (along with the adjacent edges) we
obtain a new diagram $\ul{Dyn}$, which corresponds to a semisimple
Levi subgroup $\ul G\subset G$. The restriction of $f_\lambda$ to
$\ul{Dyn}$ defines a simple representation $\ul{V}$ of the group
$\ul G$.

\begin{df}\label{Def_Chopping}
We say that a $\ul G$-representation $\ul V$ is a {\it chopping} of
a $G$-representation $V$, if $\ul V$ is obtained from $V$ via the
above construction.
\end{df}
\begin{rem}Note that a chopping of a fundamental representation is either
fundamental, or trivial one-dimensional.\end{rem}

The Dynkin diagram determines the Weyl group ${\mc W}$. We denote by
$w_0$ the longest element of ${\mc W}$ with respect to the Bruhat
order. The weights of the form $w\lambda$, with $w\in{\mc W}$, are
called {\it extreme weights} of the module $V(\lambda)$ and the
corresponding weight vectors are called extreme weight vectors. For
any $w\in{\mc W}$ we denote by $v^{w\lambda}$ the unique up to
scaling vector of weight $w\lambda$. The weight $w_0\lambda$ is
called {\it the lowest weight} of $V(\lambda)$.

For $\lambda_1, \lambda_2\in\Lambda^+$, $G$-module
$V(\lambda_1)\otimes V(\lambda_2)$ contains a unique up to scaling
vector of weight $\lambda_1+\lambda_2$. This vector is contained in
a simple $G$-submodule, which is isomorphic to
$V(\lambda_1+\lambda_2)$. We call this submodule the {\it Cartan
component} of $V(\lambda_1)\otimes V(\lambda_2)$. It is well known
that the Cartan component does not depend on a choice of Borel
subgroup $B\subset G$.

Let us fix $\lambda\in\Lambda^+$ and put $V=V(\lambda)$ and
$\PP=\PP(\lambda):=\PP(V(\lambda))$. The group $G$ acts on the
projective space $\PP$ and has a unique closed orbit therein,
namely, the orbit through the highest weight line, to be denoted by
$\XX=\XX(\lambda):=G[v^\lambda]$. We have $\XX=G/P$, where
$P$ denotes the stabilizer of $[v^\lambda]\in\PP$ in $G$.
This $P$ is a standard parabolic subgroup, i.e. a closed subgroup of
$G$ containing the fixed Borel subgroup $B$. The cosets of $G$ by
parabolic subgroups are called the flag varieties of $G$. Thus we
have an equivariantly embedded flag variety $\XX=G/P\subset
\PP$. In fact, all equivariantly embedded homogeneous
projective varieties are obtained in this fashion. Note, that the
variety $\XX$ is the set of highest weight vectors with respect to
all possible choices of Borel subgroups $B\subset G$.


The irreducibility of $V$ implies that $\XX$ spans
$\PP$. Hence, we have well defined rank and border rank
functions on $\PP$ with respect to $\XX$, as well as secant
varieties $\sigma_r(\XX)\subset \PP$. Since the group $G$ acts on
$V$ by invertible linear transformations, it follows immediately
that rank and border rank are $G$-invariant functions. Hence, the
rank sets $\XX_r$ and the secant varieties $\sigma_r(\XX)$ are
preserved by $G$.

\begin{df}
Let $V=V(\lambda)$, with $\lambda\in\Lambda^+$, be an irreducible
representation of a reductive linear algebraic group $G$. Let
$\XX=\XX(\lambda)$ be the unique closed $G$-orbit in
$\PP=\PP(\lambda)$. The $G$-module $V$ is called {\it
rs-continuous} (resp. {\it r-discontinuous}, $i$-{\it continuous}, $i$-{\it discontinuous}), if the
variety $\XX\subset\PP$ is rs-continuous (resp. r-discontinuous, $i$-continuous, $i$-discontinuous).
\end{df}

Our goal is to classify all rs-continuous irreducible representations of
semisimple algebraic groups.

\begin{rem}In some of our constructions we consider reductive groups, rather
than semisimple groups, just because this simplifies some steps. The
actions of $G$ and $G'$ (the commutant of $G$) on $\PP$ coincide and we are concerned with
properties of the embedding $\XX\subset\PP$. Thus the classification
of rs-continuous representations of reductive groups can be easily obtained
from the one for semisimple groups.\end{rem} Below we assume that
$\lambda\ne0$ (this corresponds to the inequality $\dim V\ge 2,~
V=V(\lambda),$ assumed in Theorem~\ref{Tkltame}).

We denote by $A_n, B_n, C_n, D_n, E_6, E_7, E_8, F_4, G_2$ the
simple simply connected algebraic groups with the corresponding
Dynkin diagrams. We denote by $\frak a_n, \frak b_n, \frak c_n,
\frak d_n, \frak e_6, \frak e_7, \frak e_8, \frak f_4, \frak g_2$
the corresponding Lie algebras.

\section{Plan of proof of Theorem~\ref{Tkltame}}\label{SprT}
Theorem~\ref{Tkltame} follows from a number of propositions and
theorems proved throughout the article. Thus the plan explains the
role of different parts of this text.

In Proposition \ref{Lh3}, we prove that, if an irreducible
representation $V(\lambda)$ with highest weight $\lambda$ is rs-continuous,
then $h(\lambda)<3$. In Proposition \ref{Llist}, we classify all
irreducible rs-continuous modules $V(\lambda)$ with $h(\lambda)=2$.

The rs-continuous fundamental representations (i.e. irreducible $G$-modules
$V(\lambda)$ with $h(\lambda)=1$) are classified in
Theorems~\ref{Pftamecl} and~\ref{Pftameex}. This is done in a case
by case study. We consider separately representations of classical
groups (Theorem~\ref{Pftamecl}), and representations of exceptional
groups (Theorem~\ref{Pftameex}). This completes the proof of
Theorem~\ref{Tkltame}.

Let us say a few words about the proofs of
Theorems~\ref{Pftamecl} and~\ref{Pftameex}. The rs-continuous representations are considered individually and for each case we provide specific arguments.
Most representations are r-discontinuous and thus we have to check
r-discontinuity of a huge amount of cases. The number of cases to be
considered is greatly reduced by Proposition \ref{Predl}, where we
show that, if a $\ul{G}$-representation $\ul V$ is a chopping of a
$G$-representation $V$ and $\ul{V}$ is r-discontinuous, then $V$ is also r-discontinuous. Thus,
it suffices to check r-discontinuity directly for only few basic cases, then we are able to
deduce r-discontinuity for most of the fundamental representations.

\section{Non-fundamental rs-continuous modules}\label{Sgm}
The goal of this section is to classify all non-fundamental rs-continuous
modules. First of all we provide in Lemma~\ref{L2w} of
Subsection~\ref{SSsx} a way to construct exceptional vectors in
$V(\lambda)$. Using this construction we show in
Proposition~\ref{Lh3} of Subsection~\ref{Sred} that, if a $G$-module
$V(\lambda)$ is 2-continuous, then it is either a fundamental $G$-module
(i.e. $h(\lambda)=1$) or is a Cartan component in a tensor product
of two fundamental $G$-modules (i.e. $h(\lambda)=2$). Further on, in
Proposition \ref{Lh2}, we shall show that, in the latter situation,
each of the fundamental modules satisfies a very strict condition,
related to the notion of HW-density introduced in
Definition~\ref{Dhwd}. Using the explicit description of all
HW-dense modules which we present in Corollary~\ref{Cfew}, we are
able to complete in Proposition~\ref{Llist} a classification of
non-fundamental rs-continuous $G$-modules $V(\lambda)$ (i.e. modules
$V(\lambda)$ such that $h(\lambda)=2$).


\subsection{The varieties $\sigma_2(\XX(\lambda))$ and
$\XX_2(\lambda)$}\label{SSsx}
In some sense, in this article we study the
difference between the variety $\sigma_2(\XX(\lambda))$ and its open subset
$\XX_2(\lambda)$. Here we collect some of their basic
features. First we recall that, if the generic rank of
$V(\lambda)$ is greater than 1 (i.e. if
$\PP(\lambda)\ne\XX(\lambda)$), then
$$\sigma_2(X(\lambda))=\overline{X_2(\lambda)}.$$
We provide some kind of explicit description of the elements of
$\XX_2(\lambda)$ (Lemma~\ref{Lzak?}) and exhibit some set of
elements of $\sigma_2(\XX(\lambda))$, which tend to be exceptional
(Lemma~\ref{L2w}).
\begin{lemma}~\label{Lzak?} Let $V(\lambda)$ be an irreducible representation of a reductive group $G$. Then

a) any pair $([v_1], [v_2])\in \XX(\lambda)\times \XX(\lambda)$ is $G$-conjugate to a pair $([v^\lambda],
[v^{w\lambda}])$ for some $w\in\mathcal W$,

b) any element of $\XX_2(\lambda)$ is conjugate to
$[v^\lambda+v^{w\lambda}]$ for some $w\in\mc W$.\end{lemma}
\begin{proof}Fix $v_1, v_2\in X(\lambda)$ such that $[v_1]\ne[v_2]$. Let $B_1, B_2$ be Borel subgroups of $G$ such that $v_1, v_2$ are
the corresponding to $B_1, B_2$ highest weight vectors. It is known
that $B_1\cap B_2$ contains a maximal torus $T_{12}$ of $G$ and
there exists $w\in\mc W_{12}:=N_G(T_{12})/T_{12}$ such that
$B_1=wB_2$. Thus $([v_1], [v_2])$ is conjugate to $([v^{\lambda}],
[v^{w\lambda}])$ for some $w\in\mathcal W$. This
completes the proof of part a).

To prove part b) we observe that any element of $X_2(\lambda)$ is a
sum $v_1+v_2$ for some $v_1, v_2\in X(\lambda)$ such that $v_1\ne
v_2$. According to part a) the pair $([v_1], [v_2])$ is conjugate to
$([v^{\lambda}], [v^{w\lambda}])$ for some $w\in\mathcal W$. Thus
$[v_1+v_2]$ is conjugate to $[av^\lambda+bv^{w\lambda}]$ for some
non-zero $a,b\in \FF$. As $[v_1]\ne [v_2]$, $\lambda\ne w\lambda$.
Hence $[v^\lambda+v^{w\lambda}]$ is conjugate to
$[av^\lambda+bv^{w\lambda}]$ for any non-zero $a,b\in \FF$. This
completes the proof of b).
\end{proof}
\begin{coro} Let $w_0$ be the longest Weyl group element. The orbit $G[v^\lambda+v^{w_0\lambda}]$ is open in both varieties
$\sigma_2(\XX(\lambda))$ and $\XX_2(\lambda)$. (see
also{~\cite[Ch. III, Thm 1.4]{Zak-Book}})\end{coro}
\begin{lemma}\label{L2w}Fix $x\in X(\lambda)$ and $t\in\frak g$. Then $[x+tx]\in\sigma_2(\XX(\lambda))$.\end{lemma}
\begin{proof}By definition $\XX_2(\lambda)\cup\XX(\lambda)$ is the union of lines going
through pairs of points of $\XX(\lambda)$. Thus the tangent space
$\T_xX(\lambda)\subset V$ to $X(\lambda)$ in $x$ belongs to
$\overline{X_2(\lambda)}=\sigma_2(X(\lambda))$. On the other
hand, $x+tx$ belongs to $\T_xX(\lambda)$ as $tx$ is tangent to
$X(\lambda)$. This completes the proof.
\end{proof}

\subsection{HW-density and 2-continuity}
\label{Sred} In this subsection, we analyze the notion of 2-continuity via the notion of HW-density given in
Definition~\ref{Dhwd}. This analysis allows to find out all rs-continuous
modules which are not fundamental. Notice that 2-continuity is, a priori, weaker than rs-continuity, but is much
simpler to check. A posteriori, it turns out that rs-continuity and 2-continuity are equivalent for the class of
homogeneous projective varieties considered in this paper.

We proceed in the following way. We prove that, if an irreducible
$G$-module $V$ is 2-continuous, then it is either a fundamental module of
$G$ (i.e. only one mark of the highest weight of $V$ is distinct
from zero and this mark equals 1) or is a Cartan component of the
tensor product of two fundamental modules of $G$, see
Proposition~\ref{Lh3}. In the latter case we prove that both
fundamental modules in the product have to be HW-dense, see
Proposition~\ref{Lh2}. It turns out that HW-density is a very strict
condition as we show in Corollary~\ref{Cfew}. This result leads to
Proposition~\ref{Llist}, which lists all 2-continuous $G$-modules, which
are not fundamental.

We start with the definition of HW-density followed by the
statements of the results. The proofs of these results are given
below until the end of the section.
\begin{df}\label{Dprehwd}Let $\XX$ be a smooth subvariety of a projective space $\PP$. We say that $\XX$ is {\it HW-dense}, if for any point $x_1\in\XX$ there exists an open subset $U$ of the tangent space to $\XX$ at $x_1$ such that for all $v\in U$ there exists $x_2\in\XX$ such that $v\in \APbr{x_1, x_2}$.\end{df}
\begin{df}\label{Dhwd} We say that a simple $G$-module $V(\lambda)$ is {\it
HW-dense}, if $\XX(\lambda)$ is HW-dense~in~$\PP(\lambda)$.\end{df}
We shall prove below in this subsection the following criterion of
HW-density.
\begin{lemma}\label{Lfew2}The $G$-module $V(\lambda)$ is HW-dense if and
only if one of the following equivalent conditions holds:

(a) The set $X(\lambda)$ of highest weight vectors is dense in
$V(\lambda)$.

(b) All non-zero vectors of $V$ are highest weight vectors with
respect to some choice of a Borel subgroup of $G$.

(c) $G$ acts transitively on $\PP(V)$.\end{lemma}
\begin{coro}\label{Cfew} Let $V$ be an effective fundamental HW-dense $G$-module. Then $(G, V)$ is isomorphic to $(\Sp(V), V)$, $(SL(V), V)$ or $(SL(V),V^*)$.\end{coro}
Now, we formulate Propositions~\ref{Lh3},~\ref{Lh2} announced at the
beginning of Section~\ref{Sgm}.

\begin{prop}\label{Lh3}Assume that $V(\lambda)$ is rs-continuous. Then $h(\lambda)<3$.\end{prop}
\begin{prop}\label{Lh2}Let $\lambda_1, \lambda_2\in\Lambda$ be non-zero weights. Assume that $V(\lambda_1+\lambda_2)$ is 2-continuous. Then both $V(\lambda_1), V(\lambda_2)$ are HW-dense.\end{prop}
Proofs of Propositions~\ref{Lh3},~\ref{Lh2} are presented below in
this subsection. Corollary~\ref{Cfew} and Proposition~\ref{Lh2}
immediately imply the following proposition.
\begin{prop}\label{Llist}Assume that $V(\lambda)$ is an effective 2-continuous module and $h(\lambda)=2$. Then $(G, V(\lambda))$ appears in the following list:
\begin{center}1) $(SL(V_1)\times SL(V_2), V_1\otimes V_2)$;\hspace{10pt}2) $(SL(V_1)\times\Sp(V_2), V_1\otimes V_2)$;\hspace{10pt}3) $(\Sp(V_1)\times\Sp(V_2), V_1\otimes V_2)$;\\4) $(SL(V),~$S$^2V)$;\hspace{10pt}5) $(SL(V), \mathfrak{sl}(V))$;\hspace{10pt}6) $(\Sp(V), \mathfrak{sp}(V))\cong(Sp(V),$ S$^2V)$\end{center}(here $\mathfrak{sl}(V), \mathfrak{sp}(V)$ denote the adjoint modules of the corresponding groups).\end{prop}
All modules listed in Proposition~\ref{Llist} are rs-continuous. Cases 1) and
4) of Lemma~\ref{Llist} are known to be rs-continuous. Cases 2), 3) have the
same secant varieties as case 1), and hence are also rs-continuous. Cases 5)
and 6) are rs-continuous due to~\cite{Kaji-SecAdjVar}. Therefore
Proposition~\ref{Llist} explicitly lists all non-fundamental rs-continuous
modules.

The rest of the current subsection is dedicated to the proofs of
Propositions~\ref{Lh3},~\ref{Lh2} and
Lemma~\ref{Lfew2}.\label{SSplh3} We need the following lemma for the
$GL(V_1)\times GL(V_2)\times GL(V_3)$-module $V_1\otimes V_2\otimes
V_3$, where $V_1, V_2, V_3$ are finite-dimensional vector spaces.
\begin{lemma}\label{Lv3}Let $x_i, y_i$ be linearly independent vectors in $V_i$ for $i=1, 2, 3$. Then \begin{center}$T:=x_1\otimes x_2\otimes x_3+y_1\otimes x_2\otimes x_3+x_1\otimes y_2\otimes x_3+x_1\otimes x_2\otimes y_3\ne v_1\otimes v_2\otimes v_3+w_1\otimes w_2\otimes w_3$\end{center}for all $v_i, w_i\in V_i~(i=1, 2, 3)$. In other words, $\rk T>2$.\end{lemma}
\begin{proof}Assume on the contrary that \begin{center}$T= v_1\otimes v_2\otimes v_3+w_1\otimes w_2\otimes w_3$\end{center}for some $v_i, w_i\in V_i, i=1, 2, 3$.
We have $V_1\otimes V_2\otimes V_3\cong\Hom(V_2^*\otimes V_3^*,
V_1)$. Therefore for any $x\in V_1\otimes V_2\otimes V_3$ we can
define $\Ima_1(x)\subset V_1$ as the image of the corresponding
homomorphism from $\Hom(V_2^*\otimes V_3^*, V_1)$. Similarly we
define $\Ima_2(x)$ and $\Ima_3(x)$. We have
\begin{center}$\Ima_i(T)=\APbr{x_i, y_i}, i=1, 2, 3$.\end{center}On the other hand, if $v_1\otimes v_2\otimes v_3+w_1\otimes w_2\otimes w_3\ne 0$,
\begin{center}$\Ima_i(v_1\otimes v_2\otimes v_3+w_1\otimes w_2\otimes w_3)\subset\APbr{v_i, w_i} (i=1, 2, 3)$.\end{center}
Therefore, $\APbr{x_i, y_i}=\APbr{v_i, w_i}, i=1, 2, 3$. Hence,
without loss of generality, we may assume that $$V_i=\APbr{x_i,
y_i}=\APbr{v_i, w_i} (i=1, 2, 3),$$ i.e. that $\dim V_i=2~(i=1, 2,
3)$. For two-dimensional spaces $V_1, V_2, V_3$ this lemma is well
known; see e.g.~\cite{Landsberg-2012-book}.\end{proof} Note that for
$(G, V(\lambda))=(SL(V_1)\times SL(V_2)\times SL(V_3), V_1\otimes
V_2\otimes V_3)$ we have $h(\lambda)=3$. Moreover, Lemma~\ref{Lv3}
is essentially a particular case of Proposition~\ref{Lh3}.
\begin{proof}[Proof of Proposition~\ref{Lh3}]Assume on the contrary that $h(\lambda)\ge 3$. Then there exist non-zero weights $\lambda_1, \lambda_2, \lambda_3\in\Lambda^+$ such that $\lambda=\lambda_1+\lambda_2+\lambda_3$. Then \begin{center}$v^{\lambda_1}\otimes v^{\lambda_2}\otimes v^{\lambda_3}\in V(\lambda_1)\otimes V(\lambda_2)\otimes V(\lambda_3)$\end{center} is a highest weight vector of weight $\lambda=\lambda_1+\lambda_2+\lambda_3$. Therefore the smallest $G$-submodule of $V(\lambda_1)\otimes V(\lambda_2)\otimes V(\lambda_3)$ containing $v^{\lambda_1}\otimes v^{\lambda_2}\otimes v^{\lambda_3}$ is isomorphic to $V(\lambda)$. We identify $V(\lambda)$ with this submodule of $V(\lambda_1)\otimes V(\lambda_2)\otimes V(\lambda_3)$ and set $$v^{\lambda}:=v^{\lambda_1}\otimes v^{\lambda_2}\otimes
v^{\lambda_3}.$$
By Lemma~\ref{L2w} we have
$$v^\lambda+tv^\lambda\in\overline{X_2(\lambda)}(=\sigma_2(X(\lambda)))$$
for any $t\in\frak g$. Therefore
\begin{equation}\ul{\rk}_{\XX(\lambda)}(v^\lambda+tv^\lambda)\le2\label{Ebrk2}\end{equation}for any
$t\in\frak g$. By the Leibnitz rule we
have\begin{center}$T_\lambda:=v^\lambda+tv^\lambda=v^{\lambda_1}\otimes
v^{\lambda_2}\otimes v^{\lambda_3}+tv^{\lambda_1}\otimes
v^{\lambda_2}\otimes v^{\lambda_3}+v^{\lambda_1}\otimes
tv^{\lambda_2}\otimes v^{\lambda_3}+v^{\lambda_1}\otimes
v^{\lambda_2}\otimes tv^{\lambda_3}$\end{center}(note that
$T_\lambda$ is of the form $T$ of Lemma~\ref{Lv3}). We claim that
\begin{center}$v^\lambda+tv^\lambda\not\in X_2(\lambda)\cup X(\lambda)\cup \{0\}$\end{center}
for all $t\in U$ from some open subset $U$ of $\frak g$. As
$\lambda_i\ne 0$, there exists some open subset $U\subset\frak g$
such that $[tv^{\lambda_i}]\ne[v^{\lambda_i}]~(i=1, 2, 3)$ for all
$t\in U$. We fix $t\in U$. We claim that
\begin{equation}\rk(v^\lambda+tv^\lambda)\ge3.\label{Erk3}\end{equation} Assume on the contrary that
$v^\lambda+tv^\lambda\in X_2(\lambda)\cup X(\lambda)\cup 0$, then
\begin{center}$v^\lambda+tv^\lambda=g_1(v^{\lambda_1})\otimes
g_1(v^{\lambda_2})\otimes
g_1(v^{\lambda_3})+g_2(v^{\lambda_1})\otimes
g_2(v^{\lambda_2})\otimes g_2(v^{\lambda_3})$\end{center}for some
$g_1, g_2\in G$ and thus
$$T_\lambda=v^\lambda+tv^\lambda=v_1\otimes v_2\otimes
v_3+w_1\otimes w_2\otimes w_3$$for some $v_1, v_2, v_3, w_1, w_2,
w_3\in V(\lambda)$. This contradicts the statement of
Lemma~\ref{Lv3}. Comparing~(\ref{Ebrk2}) and~(\ref{Erk3}) we see
that $v^\lambda+tv^\lambda$ is an exceptional vector of $V(\lambda)$
and thus $V(\lambda)$ is 2-discontinuous.\end{proof}
\begin{proof}[Proof of Proposition~\ref{Lh2}]We use notation analogous to the one in Proposition~\ref{Lh3}. Fix $\lambda:=\lambda_1+\lambda_2$. Set $$v^\lambda:=v^{\lambda_1}\otimes v^{\lambda_2}\in V(\lambda_1)\otimes V(\lambda_2).$$ This defines a canonical embedding $V(\lambda)\to V(\lambda_1)\otimes V(\lambda_2)$. As
$\lambda_i\ne 0$, there exists some open subset $U\subset\frak g$
such that $[tv^{\lambda_i}]\ne[v^{\lambda_i}]~(i=1, 2)$ for all
$t\in U$. We fix $t\in U$. Repeating the argument
preceding~(\ref{Ebrk2}) we deduce that~(\ref{Ebrk2}) holds in the present
notation. Hence, since $V(\lambda)$ is 2-continuous we have
\begin{center}$v^\lambda+tv^\lambda=g_1v^\lambda+g_2v^\lambda$ or
$v^\lambda+tv^\lambda=g_1v^\lambda$\end{center} for some $g_1,
g_2\in G$. Then
\begin{center}$\Ima_i(v^\lambda+tv^\lambda)=\APbr{v^{\lambda_i}, tv^{\lambda_i}}\;(i=1, 2),$\end{center}(for the definition of $\Ima_i$ see proof of Lemma~\ref{Lv3}) and, if $g_1v^\lambda+g_2v^\lambda\ne 0$,\begin{center}$\Ima_i(g_1v^\lambda+g_2v^\lambda)=\APbr{g_1v^{\lambda_i}, g_2v^{\lambda_i}}\;(i=1, 2)$ or $\Ima_i(g_1v^\lambda+g_2v^\lambda)=\APbr{g_1v^{\lambda_i}}$.\end{center}
Hence $g_1v^{\lambda_i}, g_2v^{\lambda_i}\in\APbr{v^{\lambda_i},
tv^{\lambda_i}}$ and either
\begin{center}$[g_1v^{\lambda_i}]\ne[v^{\lambda_i}]$ or
$[g_2v^{\lambda_i}]\ne[v^{\lambda_i}]~(i=1,
2)$.\end{center}Therefore both $V(\lambda_1)$ and $V(\lambda_2)$ are
HW-dense. This completes the proof.\end{proof}

To prove Lemma~\ref{Lfew2} we need two technical lemmas. The first gives a reformulation of
Definition~\ref{Dhwd}.
\begin{lemma}\label{Ldhwd}A simple $G$-module $V(\lambda)$ is HW-dense
if and only if there exists an open subset $U\subset\frak g$ such
that, for all $t\in U$, there exists an element $$v\in
X(\lambda)\cap \APbr{ v^\lambda, tv^\lambda}$$ such that $[v]\ne
[v^\lambda]$ (note that if such an element $v$ exists, then
$[tv^\lambda]\ne [v^\lambda]$).\end{lemma}
\begin{proof}This is a fairy easy exercise for Lie algebras-Lie groups formalism. We omit it.\end{proof}

\begin{lemma}\label{Lroot}Fix $v_1\in X(\lambda)$. Assume that there exists $v_2\in X(\lambda)$ such that $v_2\in\APbr{v_1, \frak gv_1}$. Then all non-zero vectors of $\APbr{v_1, v_2}$ belong to $X(\lambda)$.\end{lemma}
\begin{proof}Without loss of generality we assume that
$[v_1]\ne[v_2]$. Then, by Lemma~\ref{Lzak?}a), the pair $(v_1, v_2)$
is conjugate to the pair $(v^\lambda, v^{w\lambda})$ for some
$w\in\mathcal W$ and thus we can assume that
\begin{center}$v_1=v^\lambda$ and $v_2=v^{w\lambda}$\end{center}for
the fixed maximal torus $T\in G$. The space $\frak gv_1$ is clearly
$T$-invariant and the weights of this space form a subset of the set
$\lambda+\Delta$ (this is a point-wise sum). As
$v^{w\lambda}\in\frak gv^\lambda$, we
have\begin{center}$w\lambda=\lambda+\beta$ for some
$\beta\in\Delta$\end{center} (note that $w\lambda\ne\lambda$ as
$[v_1]\ne[v_2]$). Let $SL_2(\beta)$ be the $T$-stable
$SL_2$-subalgebra corresponding to the root $\beta\in\Delta$. Then the
space $\APbr{v^\lambda, v^{w\lambda}}$ is a two-dimensional simple
$SL_2(\beta)$-module and thus any two non-zero elements of
$\APbr{v_1, v_2}$ are $SL_2(\beta)$-conjugate, and hence
$G$-conjugate. Therefore all non-zero elements of $\APbr{v_1, v_2}$
belong to $X(\lambda)$.\end{proof}

\begin{proof}[Proof of Lemma~\ref{Lfew2}] The equivalence of conditions (a), (b), (c) is clear. It is
also immediate to verify that each of these conditions implies HW-density.
It remains to show that, if the module $V(\lambda)$ is HW-dense, then it satisfies
condition (a).

Assume that $V(\lambda)$ is
HW-dense. Then according to Lemma~\ref{Ldhwd} and Lemma~\ref{Lroot}
there exists an open subset $U\subset\frak g$ such that for any
non-zero $a\in\mathbb F$ and any $t\in U$ we have\begin{center}$v^\lambda+a
tv^\lambda\in X(\lambda)$,\end{center}i.e. we have that
$X(\lambda)\cap\APbr{v^\lambda, \frak gv^\lambda}$ is a dense subset
of $\APbr{v^\lambda, \frak gv^\lambda}$. Note that
$v^\lambda\in\frak gv^\lambda$, as $\lambda\ne0$, and hence
$$\APbr{v^\lambda, \frak gv^\lambda}=\APbr{\frak gv^\lambda}.$$We
have\begin{center}$\dim X(\lambda)=\dim Gv^\lambda=\dim\frak
gv^\lambda$\end{center} and therefore $$\overline{X(\lambda)}=\frak
gv^\lambda.$$On the other hand $$V(\lambda)=\APbr{X(\lambda)}$$ and
hence $$V(\lambda)=\overline{X(\lambda)}=\frak
gv^\lambda.\label{Ehwd-1}$$\end{proof}

\section{Restriction to a Levi subgroup}\label{SSrls}

The main result of this section is Proposition~\ref{Predl}. This proposition provides a strong sufficient condition for r-discontinuity of a representation. We will apply this proposition to study rs-continuity/r-discontinuity of fundamental representations in the subsequent sections of this article.

\begin{prop}\label{Predl}
Let $V$ be an irreducible $G$-module and $\ul V$ be a $\ul
G$-module, which is a chopping of $V$. If $\ul V$ is r-discontinuous, then $V$
is r-discontinuous.
\end{prop}
This proposition is an immediate corollary of Proposition~\ref{Prop
Levi Reduction}. To state Proposition~\ref{Prop Levi Reduction} we
need more notation.

Recall that we have fixed Cartan and Borel subgroups $H\subset
B\subset G$ and that $\Pi$ denotes the corresponding set of simple
roots. Let $\ul \Pi \subset \Pi$ be a subset. Then $\ul
\Delta=\Delta\cap \APbr{\ul\Pi}$ is a root system having $\ul \Pi$
as a set of simple roots and $\ul\Delta^\pm=\ul\Delta\cap\Delta^\pm$
as sets of positive and negative roots. Further, let $\ul{\mk
g}={\mk h}\oplus(\oplus_{\alpha\in\ul\Delta}){\mk g}^\alpha$. Then
$\ul{\mk g}$ is a reductive subalgebra of ${\mk g}$; we call
subalgebras of this form (reductive) Levi subalgebras. Let $\ul
G\subset G$ be the corresponding Levi subgroup. We shall add
underline to denote the attributes of $\ul G$ with the notational
conventions already introduced for $G$.

Note that $G$ and $\ul G$ have a common Cartan subgroup $H$ and
hence have the same weight lattice $\Lambda$. However, the dominant
Weil chambers do not coincide, unless $\ul\Pi=\Pi$, a case which is
of no use for us. We have an inclusion
$\Lambda^+\subset\ul\Lambda^+$, so a weight $\lambda\in\Lambda^+$
can be regarded as a dominant weight for both $G$ and $\ul{G}$.
Furthermore, since $\ul B=B\cap\ul G$, the $B$-highest weight
vectors are also $\ul B$-highest weight vectors.

Fix $\lambda\in\Lambda^+$. There is a $\ul G$-equivariant inclusion
of the corresponding representations
$$
\ul{V} = \ul V(\lambda) = {\mk U}(\ul{\mk g}) v^\lambda \subset
V(\lambda) = V \;,
$$
where $v^\lambda$ denotes the $B$-highest weight vector in
$V(\lambda)$. Let $\ul\XX$ denote the unique closed $\ul G$-orbit in
$\PP(\ul V)$ and, as before, let $\XX$ denote the unique closed
$G$-orbit in $\PP(V)$. We have $$\ul\XX=\ul G[v^\lambda]\subset
G[v^\lambda] =\XX.$$ For points in $\PP(\ul V)$, we have two well
defined rank functions ${\rm rk}_{\ul \XX}$ and ${\rm rk}_\XX$
(notice that $\ul V$ is a chopping of $V$ according to
Definition~\ref{Def_Chopping}). We would like to compare these
functions and prove the following.
\begin{prop}\label{Prop Levi Reduction}
Let $\ul{G}\subset G$ and $\ul{V}\subset V$ be as above. If
$[\psi]\in\PP(\ul V)$, then ${\rm rk}_{\ul{\XX}}[\psi] = {\rm
rk}_\XX [\psi]$.
\end{prop}

\begin{proof}
First, observe that the multiplicity of the $\ul G$-module $\ul
V(\lambda)$ in $V(\lambda)$ is 1. This holds because $\ul
V(\lambda)$ has a weight vector with weight $\lambda$ and this
weight has multiplicity 1 in $V(\lambda)$. Consequently, there is a
well-defined $\ul G$-equivariant projection
$$
\pi : V \twoheadrightarrow \ul{V} \;.
$$

Let $P\subset G$ be the parabolic subgroup containing $B$ and having
$\ul G$ as a Levi component; the roots of $P$ are
$\Delta^+\sqcup\ul\Delta^-$. Let $N_P$ be the unipotent radical of
$P$; the roots of $N_P$ are $\Delta^+\setminus\ul\Delta^+$. Then
$N_P$ acts trivially on $\ul V$.




\begin{lemma}
We have $\pi(X\cup 0)=\ul X\cup 0$.
\end{lemma}

\begin{proof}Let $N_P^-$ be the nilradical of the parabolic
$P^-$ opposite to $P$, with respect to the given Cartan subgroup
$H$. In other words, $N_P^-$ is the regular unipotent subgroup of
$N^-$ with roots $\Delta(N_P^-)=-\Delta(N_P)$. We have
$$
 X  = \ol{Gv^\lambda} = \ol{P^-v^\lambda} = \ol{N_P^-(\ul Gv^\lambda)} = \ol{N_P^-\ul X} \;.
$$
Thus, to prove the lemma it is sufficient to show that for all $g\in
N_P^-$ and all $v\in\ul X$ we have $\pi(gv)\in\ul X$. Let $g\in
N_P^-$ and $v\in\ul X$. Since the exponential map ${\rm exp}:{\mk
n}_{\mk p}^-\rw N_P^-$ is surjective, we can write $g={\rm
exp}(\xi)$ with $\xi\in{\mk n}_{\mk p}^-$. Viewing $\xi$ as an
element of $\mk{gl}(V)$ we can write
$$
gv=(1+\xi+\frac{1}{2}\xi^2+...)v=v+\xi v+\frac{1}{2}\xi^2v+\cdots
\;.
$$
Let $\ul V'={\rm ker}(\pi)$, so that $V=\ul V \oplus \ul V'$ as $\ul
G$-modules. Then, for $\xi\in {\mk n}_{\mk p}^-$, we have $\xi(\ul
V)\subset \ul V'$. Hence $\pi(gv)=v$.
\end{proof}

Now, let $[\psi]\in\PP(\ul V)$. The inequality $\rk_{\ul \XX}[\psi]\geq{\rk_{\XX}}[\psi]$ is immediate. Let $r={\rm rk}[\psi]$
and $$\psi=v_1+...+v_r$$ be a minimal expression, with
$[v_j]\in\XX$. Then we have $$\psi=\pi(\psi)=\pi(v_1)+...+\pi(v_r)$$
and, according to the above lemma, $[\pi(v_j)]\in\ul\XX\cup 0$ (this
is a set). Hence $\rk_{\ul\XX}[\psi]\leq\rk_{\XX}[\psi]$ and so
\begin{center}$\rk_{\ul\XX}[\psi]=\rk_{\XX}[\psi].$\end{center}
\end{proof}

\section{Fundamental representations (classical groups)}\label{Sfrepcl}

In this subsection we prove Theorem~\ref{Tkltame} for fundamental
modules of classical groups, i.e. we prove Theorem~\ref{Pftamecl}.
The result follows directly from Propositions~\ref{Prop Fund
SL},~\ref{Prop Fund SO},~\ref{Prop Fund Sp} of
Subsections~\ref{SSsl},~\ref{SSso},~\ref{SSsp}, where we consider
the cases of $SL_n$, $SO_n$, $Sp_{2n}$, respectively.
\begin{theorem}\label{Pftamecl}
Let $V(\lambda)$ be a fundamental module of a simple classical group
$G$. Then $V(\lambda)$ is rs-continuous if and only if the pair $(G,
V(\lambda))$ appears in the following table.
\begin{equation}\begin{tabular}{|c|}\hline$\begin{tabular}{c|c|c}Group
$G$&Representation $V$&Highest weight of $V$\end{tabular}$\\
\hline Classical groups\\\hline
$\begin{tabular}{c|c|c}$SL_n$&$\CC^n, (\CC^n)^*, (\Lambda^2\CC^n),
(\Lambda^2\CC^n)^*$&$\pi_1, \pi_{n-1}, \pi_2,
\pi_{n-2}$\\\hline$SO_n$&$\CC^n,
RSpin_n(n\le10)$&$\begin{tabular}{c}$\pi_1, \pi_{\frac n2} (2\mid
n),$\\$\pi_{\frac n2-1} (2\mid n), \pi_{\frac{n-1}2} (2\nmid
n)$\end{tabular}$\\\hline$SP_{2n}$&$\CC^{2n},
\Lambda^2_0\CC^{2n}$&$\pi_1,
\pi_2$\\\end{tabular}$\\\hline\end{tabular}\hspace{10pt},\end{equation}
where the notation is the same as in Theorem~\ref{Tkltame}.

Moreover, all r-discontinuous fundamental representations of classical groups
are 2-discontinuous.\end{theorem} Our approach for classical groups is
tensor-based and we often use symmetric/antisymmetric bilinear
forms. To prove Theorem~\ref{Pftamecl} we also need some sufficient
condition of r-discontinuity for representations. Such a condition is provided in Proposition~\ref{Predl} of Section~\ref{SSrls}. In a
similar way Proposition~\ref{Predl} will be very useful in
Section~\ref{Sfrepex}, where we consider the fundamental
representations of the exceptional groups.

\subsection{$G=SL_n$}\label{SSsl}

Recall that the fundamental representations of $G=SL_n$ are obtained
as exterior tensor powers of the natural representation, i.e.
$V(\pi_k)=\Lambda^k\FF^n$, $k=1,...,n-1$. Furthermore, we have
$$(\Lambda^k\FF^n)^*=\Lambda^{n-k}\FF^n$$ as $SL_n$-modules.
\begin{prop}\label{Prop Fund SL}
The fundamental representations of $SL_n$ which are rs-continuous are exactly
$$
\CC^n, (\CC^n)^*, \Lambda^2\CC^n,
(\Lambda^2\CC^n)^*.
$$
Moreover, all r-discontinuous fundamental representations of $SL_n$ are 2-discontinuous.
\end{prop}

\begin{proof}
The closed $G$-orbit $\XX\subset\PP(\Lambda^k\FF^n)$ is the
Grassmann variety Gr$_k(\FF^n)$ under its Pl\"ucker embedding. It is
well known that a suitable isomorphism between $\Lambda^k\FF^n$ and
$\Lambda^{n-k}\FF^n$ induces an isomorphism between the respective
projective spaces, which carries Gr$_k(\FF^n)$ to Gr$_{n-k}(\FF^n)$.
Hence, for our purposes, it is sufficient to consider $k\leq n/2$.

The fact that $V(\pi_1)=\FF^n$ and $V(\pi_2)=\Lambda^2\FF^n$ are
rs-continuous is well known. In fact, in the first case we have
$\XX=\PP(\FF^n)$, so all vectors have rank 1. In the second case,
$\Lambda^2\FF^n$ can be identified with the space of skew-symmetric
$n\times n$ matrices. Such a matrix has even rank (in the usual
sense) and the $SL_n$-orbit $X$ through a highest weight vector in
$\Lambda^2\FF^n$ consists of all matrices of rank 2. A
skew-symmetric matrix $\psi$ of rank $2r$ can be written as a sum of
$r$ skew-symmetric matrices of rank 2, and so ${\rm
rk}_\XX[\psi]=r$. We can now see that the set
$$\{[\psi]\in\PP(\Lambda^2\FF^n):{\rm rk}_\XX\psi\leq r\}$$ is closed
for every $r$. This completes the argument in this case.

We now turn to the remaining cases. Due to the
duality $\Lambda^k\mathbb F^n\leftrightarrow(\Lambda^{n-k}\mathbb F^n)^*$, it
suffices to consider $n\geq 6$. Proposition \ref{Prop Levi
Reduction} implies that, to show that $\Lambda^k\FF^n$ ($3\leq k\leq
n/2$) is 2-discontinuous, it is sufficient to show that $\Lambda^3\FF^6$ is
2-discontinuous.

\begin{lemma}\label{Lemma A^3(SL_6)}
The representation of $SL_6$ on $\Lambda^3\CC^6$ is 2-discontinuous.
\end{lemma}

\begin{proof}It is shown in~\cite[Ch. III, Thm 1.4]{Zak-Book}, that
$\sigma_2(X(\Lambda^3\CC^6))=\Lambda^3\CC^6$ and therefore it is
enough to show that $\Lambda^3\CC^6$ contains a vector of rank 3 or
more. To do this we count the number of orbits of vectors of rank 0,
1, 2 and compare this number with the known number of orbits for the
action of $SL_6$ on $\Lambda^3\CC^6$, see~\cite{Sem48}
or~\cite{R07}.

By definition there is one orbit of vectors of rank 0 and
one orbit of vectors of rank 1. Let us consider the
vectors of rank 2 in $V$. We shall show that there are two orbits of such vectors.
Any vector of rank 2 can be written as
$$
\psi = v_1\wedge v_2\wedge v_3 + v_4\wedge v_5\wedge v_6 \;,
$$
with some $v_j\in{{V}}$. The first possibility is that $v_1,...,v_6$
form a basis of $\CC^6$. This is indeed the generic situation. If
suitable Borel and Cartan subgroups of $SL_6$ are chosen, the two
summands of $\psi$ are, respectively, the highest and lowest weight
vectors in ${{V}}$. The group $GL_6$ acts transitively on the set of
all bases of $\CC^6$; the group $SL_6$ acts transitively on the set
of their projective images. Thus the points of the first type form a
single $G$-orbit $\XX_2'$ which is open in $\PP(\Lambda^3\mathbb
F^6)$. We denote by $Z$ the complement to this orbit in
$\PP(\Lambda^3\mathbb F^6)$. The second possibility is to have
$$
\dim (\APbr{v_1,v_2,v_3}\cap \APbr{v_4,v_5,v_6}) = 1 \;.
$$
If this is the case, by changing the vectors if necessary, we may reduce to the situation where $v_1=v_4$ and
$$
\psi= v_1\wedge (v_2\wedge v_3 + v_5\wedge v_6) \;,\quad {\rm with}
\quad \APbr{v_2,v_3}\cap \APbr{v_5,v_6} = 0 \;.
$$
Since $v_2\wedge v_3 + v_5\wedge v_6$ has rank 2 in $\Lambda^2\CC^6$
(with respect to Gr$_2(\CC^6)$), we deduce that $\psi$ has indeed
rank 2 in ${{V}}$. The point $\phi$ does not belong to $\XX'_2$,
because the action of $GL_6$ respects linear dependencies. On the
other hand, it is also clear that $GL_6$ acts transitively on the
set $\XX_2''$ of points of this second type, and hence $SL_6$ acts
transitively on the set of their images in $\PP$. Note that
\begin{center}
if $\dim (\APbr{v_1,v_2,v_3}\cap \APbr{v_4,v_5,v_6}) > 1$, then ${\rm rk}[\psi]=1$.
\end{center}
We can conclude that there are exactly two $G$-orbits consisting of points of rank 2, namely
$$
\quad \XX_2'=\PP\setminus Z,\quad \XX_2''=Z\cap\XX_2 \;.
$$Thus there are four $SL_6$-orbits of vectors of rank 0, 1, 2. It is
known that $SL_6$ has five orbits in $\Lambda^3\CC^6$. Therefore
$\Lambda^3\CC^6$ has a unique $SL_6$-orbit of vectors of rank 3 or
more and $\Lambda^3\CC^6$ is 2-discontinuous. An example of a vector of rank
3 is (see \cite{Sem48}) $$\Lambda3 = v_1\wedge v_2\wedge v_4 +
v_1\wedge v_5\wedge v_3 + v_6\wedge v_2\wedge v_3 \;.$$
\end{proof}
\end{proof}

\subsection{$G=SO_n$}\label{SSso}
Let $\ell=rank(G)=\lfloor\frac{n}{2}\rfloor$. In this section we
prove the following proposition.
\begin{prop}\label{Prop Fund SO}
The natural representation $V(\pi_1)$ is rs-continuous.

1) If $n$ is even, then, for $j=2,...,\ell-2$, the representation
$V(\pi_j)$ is r-discontinuous.

2) If $n$ is odd, then, for $j=2,...,\ell-1$, the representation
$V(\pi_j)$ is r-discontinuous.

3) The spin representations ($V(\pi_{\ell-1})$ and $V(\pi_{\ell})$
for even $n$ and $V(\pi_\ell)$ for odd $n$) are rs-continuous if and only if
$n\leq 10$.

Moreover, all fundamental representations of $SO_n$ which are r-discontinuous
are 2-discontinuous.
\end{prop}

\begin{proof}
The first statement is well known. Indeed, the group $SO_n$ has
exactly two orbits in $\PP(\pi_1)$, namely, the quadric and its
complement. The first one consists, by definition, of points of
rank 1. The second one consists necessarily of points of rank 2.

The second and third statement in the proposition are concerned with
fundamental representations of $SO_n$, which are not the natural nor
the spin representation. We handle the two statements at once. In
the case $j=2$, the representation is actually the adjoint
representation, i.e. $V(\pi_2)=\frak{so}_n$. Here results of
\cite{Kaji-Yasukura-2000} show that $\sigma_2(\XX)\ne
\XX\sqcup\XX_2$. Thus the representation is 2-discontinuous. The remaining
cases, $j\geq 3$, are reduced to the case $j=2$ via Proposition
\ref{Prop Levi Reduction}.

Now, we turn to the last statement of the proposition, concerning
the spin representations. First, recall that, for even $n$, the
geometric properties we are concerned with are the same for the two
spin representations $V(\pi_{\ell-1})$ and $V(\pi_\ell)$. Also,
either one of these representations remains irreducible when
restricted to $Spin_{n-1}$ and, furthermore, $Spin_{n-1}$ acts
transitively on the closed orbit of $Spin_n$ in $\PP(\pi_\ell)$.
Thus, it is enough to check statement 3) for the representations
$V(\pi_\ell)$ of $Spin_{2\ell}$. Let $\XX$ denote the closed orbit
of $Spin_{2\ell}$ in $\PP(\pi_\ell)$.

It is shown, in \cite[Section 3.5]{Charlton-Thesis-1997}, that for
$2\ell=12$ the secant variety $\sigma_2(\XX)$ contains elements of
rank 3. Thus the representation $V(\pi_6)$ of $Spin_{12}$ is 2-discontinuous.
Using Proposition~\ref{Predl}, we deduce that the representation
$V(\pi_\ell)$ of $Spin_{2\ell}$ is r-discontinuous for all $\ell\geq 6$. So,
according to the remarks made earlier in this proof, the spin
representations of $Spin_n$ are 2-discontinuous for $n\geq11$.


It remains to verify that the spin representations are rs-continuous for even
$n\leq 10$. This statement easily follows from Corollary~\ref{Cz}. This completes the proof
of the proposition.

\end{proof}

\subsection{$G=Sp_{2n}$}\label{SSsp}
\begin{prop}\label{Prop Fund Sp}
The fundamental representations of $Sp_{2n}$ which are rs-continuous are exactly $V(\pi_1)$ and $V(\pi_2)$. All other fundamental representations of $Sp_{2n}$ are 2-discontinuous.
\end{prop}

\begin{proof}
The representation $V(\pi_1)$ is simply the natural representation
of $Sp_{2n}$ on $\FF^{2n}$. The action of $Sp_{2n}$ on
$\PP(\FF^{2n})$ is transitive, i.e. $\XX=\PP(\FF^{2n})$ and there is
nothing more to prove here. The representations $V(\pi_2)$ and
$V(\pi_k), k\geq3$ are considered in Lemmas \ref{Psppi2} and \ref{Lemma A^k Sp_2n}, respectively.
\end{proof}
\begin{lemma}\label{Lemma A^3 Sp_2n}
The representation $V(\pi_3)$ of $Sp_{2n}$ is 2-discontinuous for $n\ge3$.
\end{lemma}
\begin{proof}
Let $n\geq 3$ and consider $Sp_{2n}$ to be defined with respect to the skew-symmetric
form $$(z_1\wedge z_6 + z_2\wedge z_5 + z_3\wedge z_4)+ (z_7\wedge
z_8+...+ z_{2n-1}\wedge z_{2n})$$ on $\CC^{2n}$, where
$z_1,...,z_{2n}$ is the dual basis corresponding to the basis
$v_1,...,v_{2n}$ of $\CC^{2n}$.

Let $\XX$ be the set of points $[w_1\wedge w_2\wedge w_3]$ such that
$w_1, w_2, w_3\in \CC^{2n}$ span a 3-dimensional isotropic subspace
$\CC^{2n}$. Let $X$ be the affine cone over $\XX$. We set $$\Lambda^3_0\FF^{2n}:=\APbr{X}.$$ The set $\XX$ is a single $Sp_{2n}$-orbit and thus $\Lambda^3_0\FF^{2n}$ is an irreducible
$Sp_{2n}$-module. There is a unique up to scaling
$Sp_{2n}$-isomorphism between $V(\pi_3)$ of $Sp_{2n}$ and
$\Lambda^3_0\FF^{2n}$. We identify $V(\pi_3)$ with
$\Lambda^3_0\FF^{2n}$. The varieties $\XX$ and $\XX(\pi_3)$ coincide
under this identification.

Note that if $\psi\in V(\pi_3)$ has rank $3$ as a vector of the
$SL_{2n}$-module $\Lambda^3\CC^{2n}$, then $\psi$ has rank 3 or
more as a vector of the $Sp_{2n}$-module $V(\pi_3)$.

Consider the tensor $\Lambda3\in \Lambda^3\CC^6$ given at the end of
the proof of Lemma~\ref{Lemma A^3(SL_6)}. One has
\begin{center}$[v_1\wedge v_2\wedge v_4], [v_1\wedge v_5\wedge v_3],
[v_6\wedge v_2\wedge v_3]\in \XX(\pi_3)$\end{center} and thus
$\Lambda3\in\Lambda^3_0\FF^6\subset\Lambda^3_0\FF^{2n}$. Moreover
the rank of $\Lambda3$ is 3 or less. As $\Lambda3$ has rank 3 as an
element of $SL_6$-module $\Lambda^3\CC^6$, $\Lambda3$ has rank 3 as
an element of $\Lambda^3\CC^{2n}$ (see Proposition~\ref{Prop Levi
Reduction}). Hence
\begin{equation}\rk_\XX\Lambda3=3,\label{Ersp3}\end{equation} where $\Lambda3$
is considered as an element of $\Lambda^3_0\FF^{2n}$.

It is shown in~\cite[Ch. III, Thm 1.4]{Zak-Book}, that
$\sigma_2(\XX)=\PP(\Lambda_0^3\CC^6)$. Thus $\Lambda3$ has border
rank 2 or less as an element of $\Lambda^3_0\FF^6$ and hence
\begin{center}$\ul{\rk}_\XX\Lambda3\le2$,\end{center} here $\Lambda3$ is considered as an element of
$\Lambda^3_0\FF^{2n}$. Therefore the $Sp_{2n}$-module
$V(\pi_3)=\Lambda^3_0\FF^{2n}$ is 2-discontinuous.\end{proof}

\begin{lemma}\label{Lemma A^k Sp_2n}
Fix $n\geq k\geq 3$. The representation $V(\pi_k)$ of $Sp_{2n}$
is 2-discontinuous.
\end{lemma}

\begin{proof} If $n\ge k\ge3$, the Dynkin diagram $C_n$ of $Sp_{2n}$ has a unique subdiagram $\ul C_{n, k}$ of type $C_{n-k+3}$. The chopping of $\pi_k$ to this diagram equals $\pi_3$. By Lemma~\ref{Lemma A^3 Sp_2n}, $V(\pi_3)$ is not rs-continuous for $\ul G=Sp_{2(n-k+3)}$ (this group corresponds to the Dynkin diagram $\ul C_{n, k}$) and thus $V(\pi_k)$ is not an rs-continuous $Sp_{2n}$-module by Proposition~\ref{Predl}.\end{proof}

We are now going to prove that $V(\pi_2)$ is rs-continuous for $\Sp_{2n}
(n\ge2)$. To do this we set $V:=\CC^{2n}$ and fix a nondegenerate
antisymmetric bilinear form $\omega$ on $V$. Note that the second
fundamental module of $\frak{sp}(V)$ is isomorphic to the set of
vectors in $\Lambda^2V$, which are annihilated by $\omega$ (here we
consider $\omega$ as an element of $(\Lambda^2V)^*$). We denote this
space by $\Lambda^2_0V$. To complete the proof of
Proposition~\ref{Prop Fund Sp} we prove the following lemma.
\begin{lemma}\label{Psppi2}a) For any $\ul \omega\in\Lambda^2_0V$ the rank of $\ul \omega$ as a bilinear coform is twice the rank of $\ul \omega$ as a vector in an $\Sp(V)$-module.\\
b) The $\Sp(V)$-module $\Lambda^2_0V$ is rs-continuous.\end{lemma} To prove
Lemma~\ref{Psppi2} we introduce a notion related to bilinear coforms
$\ul \omega\in\Lambda^2V$. A bilinear coform $\ul \omega$ defines a
map $V^*\to V$ by $v\to \ul \omega(v, \cdot)$. We denote the image of
this map by $\Supp\ul \omega$. We have natural inclusions
\begin{center}$\Lambda^2\Supp\ul \omega\to\Lambda^2V$, \hspace{10pt} $\Lambda^2_0\Supp\ul \omega\to\Lambda^2_0V$,\end{center} and if $\ul
\omega\in\Lambda^2_0V$, then $\ul
\omega\in\Lambda^2_0\Supp\ul \omega$. Note that
$\ul \omega$ is nondegenerate as an element of $\Lambda^2\Supp\ul
\omega$ and, in particular, defines a bilinear form $\ul \omega^*$
on $\Supp\ul \omega$ (there is no canonical way to extend $\ul
\omega^*$ to the whole $V$).

Lemma~\ref{Psppi2} follows from Lemma~\ref{Lpi21} below; a proof of
Lemma~\ref{Psppi2} is presented after the proof of Lemma~\ref{Lpi21}.
\begin{lemma}\label{Lpi21}Let $\ul \omega\in\Lambda^2_0V$ be a bilinear coform of rank $2r$. Then there exist a set of elements $x_1,..., x_r, y_1,..., y_r\in\Supp\ul \omega$ such that\begin{center}$\ul \omega=x_1\wedge y_1+...+x_r\wedge y_r$, and $\omega(x_i, y_i)=0$ for all $i$.\end{center}\end{lemma}
In turn, Lemma~\ref{Lpi21} follows from Lemma~\ref{Lpi22} below; a
proof of Lemma~\ref{Lpi21} is presented after the proof of
Lemma~\ref{Lpi22}.
\begin{lemma}\label{Lpi22}Let $\ul \omega\in\Lambda^2_0V$ be a bilinear coform of rank $2r$. If $r>0$, then there exist elements $x_1, y_1\in\Supp\ul \omega$ such that\begin{center}${\rm rank}(\ul \omega-x_1\wedge y_1)=2r-2$, and $\omega(x_1, y_1)=0$,\end{center}
where rank$(\eta)$ denotes the usual rank of a bilinear coform $\eta$.\end{lemma}
In turn, Lemma~\ref{Lpi22} follows from Lemma~\ref{Lpi23} below; a
proof of Lemma~\ref{Lpi22} is presented after the proof of
Lemma~\ref{Lpi23}.
\begin{lemma}\label{Lpi23}Let $\ul \omega\in\Lambda^2_0V$ be a bilinear coform of rank $2r$. If $r>0$, then there exists an open subset $U\subset\Supp\ul\omega$ such that for any $x_1\in U$ there exists $y_1\in \Supp\ul \omega$ such that $\ul \omega^*(x_1, y_1)\ne 0$ and $\omega(x_1, y_1)=0$.\end{lemma}
\begin{proof}If a form $\omega$ is zero on $\Supp\ul \omega$, then for any non-zero $x_1\in\Supp\ul \omega$ there exists $y_1\in\Supp\ul \omega$ such that $\ul \omega^*(x_1, y_1)\ne 0$, because the form $\ul \omega^*$ is nondegenerate on $\Supp\ul \omega$. In this case $\omega(x_1, y_1)=0$, because $\omega=0$.

We assume that $\omega$ is non-zero on $\Supp\ul \omega$.
Since $\ul\omega\in\Lambda^2_0V$, the pairing of $\ul\omega$ with
$\omega$ equals 0. Thus
$[\ul\omega^*]\ne[\omega|_{\Lambda^2\Supp\ul\omega}]$. Hence, for
some open subset $U\subset\Supp\ul \omega$ and any $x_1\in U$, both
$\omega(x_1, \cdot), \ul \omega^*(x_1, \cdot)$ are non-zero and
$$[\omega(x_1,\cdot)]\ne[\ul\omega^*(x_1, \cdot)].$$ Therefore for any $x_1\in U$ there exists
$y_1\in\Supp\ul \omega$ such that $\ul \omega^*(x_1,
y_1)\ne0$ and $\omega(x_1, y_1)=0$.\end{proof}
\begin{proof}[Proof of Lemma~\ref{Lpi22}] Let $(x_1, y_1)$ be a pair as in Lemma~\ref{Lpi23}. We denote by $W_2$ the space spanned by $x_1, y_1$ and by $W_{2r-2}$ the orthogonal complement to $W_2$ in $\Supp\ul \omega$ with respect to $\ul \omega$. Thanks to the choice of $x_1, y_1$, the form $\ul \omega^*$ is nondegenerate on $W_2$ and therefore $\Supp\ul \omega=W_2\oplus W_{2r-2}$. Then $\ul \omega=\ul \omega^2+\ul \omega^{2r-2}$ for uniquely determined coforms $\ul \omega^2\in\Lambda^2 W_2$ and $\ul \omega^{2r-2}\in\Lambda^2W_{2r-2}$. We have $\ul \omega^2=\lambda(x_1\wedge y_1)=x_1\wedge(\lambda y_1)$ for some $\lambda\in\mathbb F^\times$. Therefore \begin{center}$\rk(\ul \omega-x_1\wedge (\lambda y_1))=2r-2$,\end{center} and $\omega(x_1, (\lambda y_1))=0$.\end{proof}
\begin{proof}[Proof of Lemma~\ref{Lpi21}]To prove Lemma~\ref{Lpi21} we use induction.

The $r$-th statement of the induction is: Let $\ul \omega\in\Lambda^2_0V$ be a
bilinear coform of rank $2r$. Then there exist a set of elements
$x_1,..., x_r, y_1,..., y_r\in\Supp\ul \omega$ such
that\begin{center}$\ul \omega=x_1\wedge y_1+...+x_r\wedge
y_r$, and $\omega(x_i, y_i)=0$ for all $i$.\end{center}

Basis of the induction, for $r=1$: Let $\ul \omega\in\Lambda^2_0V$ be a bilinear
coform of rank $2$. Then there exist elements $x_1, y_1\in\Supp\ul
\omega$ such that $\ul \omega=x_1\wedge y_1$, and $\omega(x_1,
y_1)=0$.

First, we check the basis of the induction. Let $x_1, y_1$ be basis of
$\Supp\ul \omega$. Then $\ul \omega=\lambda x_1\wedge y_1$ for some
$\lambda\in\mathbb F^\times$. As $\ul \omega\in\Lambda^2_0V$, we
have $\omega(\ul \omega)=\omega(\lambda x_1\wedge y_1)=\omega(x_1,
\lambda y_1)=0$. Then $\ul \omega=x_1\wedge(\lambda y_1)$ and
$\omega(x_1, \lambda y_1)$=0. Therefore we finish with the basis of
induction.

Now we prove that the $r$-th statement of the induction follows from
the $(r-1)$-th statement. We assume that the $(r-1)$-th statement holds.
According to Lemma~\ref{Lpi22} there exists $x_r, y_r\in\Supp\ul\omega$ such
that\begin{center}$\rk(\ul \omega-x_r\wedge y_r)=2r-2$,
and $\omega(x_r, y_r)=0$.\end{center} Note that $\omega(x_r\wedge
y_r)=\omega(x_r, y_r)=0$ and therefore $\ul \omega-x_r\wedge
y_r\in\Lambda^2_0V$. By hypothesis, there exist $x_1,..., x_{r-1}, y_1,...,
y_{r-1}\in\Supp(\ul \omega-x_r\wedge y_r)\subset\Supp\ul \omega$
such that\begin{center}$\ul \omega-x_r\wedge y_r=x_1\wedge
y_1+...+x_{r-1}\wedge y_{r-1}$, and $\omega(x_i, y_i)=0$
for all $i$.\end{center} This completes the proof of Lemma~\ref{Lpi21}.\end{proof}
\begin{proof}[Proof of Lemma~\ref{Psppi2}]First note that a highest weight vector of the $\Sp(V)$-module $\Lambda_0^2V$ is a wedge product of two $\omega$-orthogonal vectors of $V$. Fix a coform $\ul \omega\in\Lambda^2_0V$. A sum of $r$ vectors from the $\Sp(V)$-orbit of a highest weight vector has rank at most $2r$ as a bilinear coform. Hence the rank of $\ul \omega$ as a vector of an $\Sp(V)$-module is not less than half the rank of $\ul \omega$ as a bilinear coform. On the other hand, Lemma~\ref{Lpi21} implies that the rank of $\ul \omega$ as a vector of an $\Sp(V)$-module is not larger than half the rank of $\ul \omega$ as a bilinear coform. Therefore the rank of $\ul \omega$ as a vector of an $\Sp(V)$-module is equal to half the rank of $\ul \omega$ as a bilinear coform. This proves part a) of Proposition~\ref{Psppi2}.

The set of coforms of rank $r$ or less is closed for all $r$ and
this finishes part b).\end{proof}

\section{Fundamental representations (exceptional
groups)}\label{Sfrepex}

In this section we prove Theorem~\ref{Tkltame} for fundamental
modules of exceptional groups, i.e. we prove Theorem~\ref{Pftameex}.
Essentially, we consider case-by-case all 27 fundamental
representations of the 5 exceptional groups and provide some
arguments for each case, by which the corresponding fundamental
module is r-discontinuous or rs-continuous. The result is presented below.
\begin{theorem}\label{Pftameex}
Assume that $V(\lambda)$ is a fundamental effective $G$-module. Then
$V(\lambda)$ is rs-continuous if and only if the pair $(G, V(\lambda))$
appears in the following table.
\begin{equation}\begin{tabular}{|c|c|c|}\hline$G$&Representation $V$&Highest weight of
$V$\\\hline $E_6$&$\CC^{27}, (\CC^{27})^*$&$\pi_1,
\pi_5$\\\hline$F_4$&$\CC^{26}$&$\pi_1$\\\hline$G_2$&$\CC^7$&$\pi_1$\\\hline\end{tabular},\end{equation}
where the notation is the same as in Theorem~\ref{Tkltame}.

Moreover, all r-discontinuous fundamental representations of exceptional groups
are 2-discontinuous.

\end{theorem}
The types of arguments are presented in the following tables.

$$\begin{tabular}{|c|c|c|}\hline Symbol&Argument for being {\bf rs-continuous}& References\\\hline
SM &\begin{tabular}{c} the representation is reduced to\\ a
subminuscule representation\\\end{tabular}&
Section~\ref{Sintro},~\cite{Bucz-Lands-2011}\\\hline F4C & the
representation is equivalent to $V({\pi_1})$ of $F_4$&
Prop.~\ref{Prop Fund-1 F4}, Subsect.~\ref{SSf4t}\\\hline\end{tabular}$$

$$\begin{tabular}{|c|c|c|}\hline Symbol&Argument for being {\bf r-discontinuous}& References\\\hline
CC&\begin{tabular}{c} the representation is chopable to an r-discontinuous\\
representation of some classical group\\\end{tabular}&---\\\hline
Ad& the representation is adjoint&
Section~\ref{Sintro},~\cite{Kaji-SecAdjVar}\\\hline
AdC&\begin{tabular}{c} the representation is chopable to the
adjoint\\ representation of some exceptional
group\\ \end{tabular}&---\\
\hline
F4D & the representation is equivalent to $V({\pi_2})$ of $F_4$& Prop.~\ref{Prop Fund-2 F4}, Subsect.~\ref{SSf4w}\\
\hline
E7D& \begin{tabular}{c} the representation is equivalent to\\ the $E_7$-representation V$(\pi_1)$\\\end{tabular}&Prop.~\ref{LE7}, Subsect.~\ref{SSSE7}\\
\hline\end{tabular}.$$

In the following tables, we provide, for each fundamental
representation of each exceptional group, an argument by which it is
rs-continuous or r-discontinuous.

$$\begin{tabular}{|c|c|c|c|c|c|c|}\hline F. weights of E$_6$&$\pi_1$&$\pi_2$&$\pi_3$&$\pi_4$&$\pi_5$&$\pi_6$\\\hline Arguments&SM&CC&CC&CC&SM&CC or Ad\\\hline\end{tabular}$$
$$\begin{tabular}{|c|c|c|c|c|c|c|c|}\hline F. weights of E$_7$&$\pi_1$&$\pi_2$&$\pi_3$&$\pi_4$&$\pi_5$&$\pi_6$&$\pi_7$\\\hline Arguments&E7D&CC&CC&CC&CC&Ad&CC\\\hline\end{tabular}$$
$$\begin{tabular}{|c|c|c|c|c|c|c|c|c|}\hline F. weights of E$_8$&$\pi_1$&$\pi_2$&$\pi_3$&$\pi_4$&$\pi_5$&$\pi_6$&$\pi_7$&$\pi_8$\\\hline Arguments&Ad&CC&CC&CC&CC&CC&AdC&CC\\\hline\end{tabular}$$
$$\begin{tabular}{|c|c|c|c|c|}\hline F. weights of F$_4$&$\pi_1$&$\pi_2$&$\pi_3$&$\pi_4$\\\hline Arguments&F4C&F4D&CC&Ad\\\hline\end{tabular}$$
$$\begin{tabular}{|c|c|c|}\hline F. weights of G$_2$&$\pi_1$&$\pi_2$\\\hline Arguments&SM&Ad\\\hline\end{tabular}$$
For the representations $V(\pi_k)$ of exceptional groups $E_n (n=6,
7, 8)$, for which argument CC is applicable, chopping of $V$ in the
vertex with number $n-1$ is a 2-discontinuous representation of a classical
group of type $D_{n-1}$. To apply argument AdC one should chop
vertex with number 1. For all representations of the exceptional
group $F_4$, to apply argument CC one can always chop the vertex
with number 1.

Let us first justify arguments SM, CC, Ad, AdC.

SM) It was shown~\cite{Bucz-Lands-2011} that any subminuscule
representation is rs-continuous, i.e. that rank and border rank coincide for
such representations.

CC) According to Proposition~\ref{Predl}, if some chopping $\ul V$
of a representation $V$ is 2-discontinuous, then $V$ is 2-discontinuous.

Ad) According to~\cite{Kaji-SecAdjVar}, all adjoint representations of exceptional groups are 2-discontinuous.

AdC) According to Proposition~\ref{Predl} and Ad), if some chopping
$\ul V$ of a representation $V$ is an adjoint representation of an
exceptional group, then $V$ is 2-discontinuous.


The rest of this section is devoted to the justification of arguments F4C, F4D and E7W.





\subsection{Rs-continuity of $V(\pi_1)$ for $F_4$}\label{SSf4t}

In this subsection we prove the following.

\begin{prop}\label{Prop Fund-1 F4}
The fundamental representation $V(\pi_1)$ of $F_4$ is rs-continuous.
\end{prop}

\begin{proof}
It is known, \cite[p. 59]{Zak-Book}, that the generic rank of
$V(\pi_1)$ is three, so that $$\sigma_3(\XX(\pi_1))=\PP(V(\pi_1)).$$
In Lemmas~\ref{Lf42t} and~\ref{Lf43t} below, we show that $V(\pi_1)$
is 2- and 3-continuous, respectively, which implies that this module is
rs-continuous.
\end{proof}

\begin{lemma}\label{Lf42t}The $F_4$-module $V(\pi_1)$ is 2-continuous.\end{lemma}

\begin{lemma}\label{Lf43t}The $F_4$-module $V(\pi_1)$ is 3-continuous.\end{lemma}

The first fundamental representation $V(\pi_1)$ of $F_4$ is
26-dimensional and is the (nontrivial) representation of the smallest
possible dimension for this group. The discussion which follows
involves several representations of various groups. This would make
the notation $V(\lambda)$ ambiguous. We have chosen to denote the
representations spaces by indices corresponding to their dimension.
The set of highest weight vectors, previously denoted by
$X(\lambda)$, will be denoted by $X(V)$. We let $V_{26}$ denote the
representation space of $(F_4,V(\pi_1))$ and $X(V_{26})$ be the set
of highest weight vectors. To study $V_{26}$ we use the fact that it
can be obtained as a generic hyperplane in the smallest,
27-dimensional representation of $E_6$, which we denote by
$V_{27}=(E_6,V(\pi_1))$. We summarize some known results in the
following lemma.

\begin{lemma}\label{LE_6V1}

{\rm (i)} The algebra of $E_6$-invariant polynomials on $V_{27}$ is
polynomial in one generator of degree 3, i.e.
$\CC[V_{27}]^{E_6}=\CC[DET]$, where $DET\in$S$^3(V_{27}^*)$.

{\rm (ii)} The orbits of $E_6$ in $V_{27}$ are the
following:\begin{center}$0$, $X(V_{27})$, $\{DET=0\}\setminus
\ol{X(V_{27})}$, $\{DET=a\}$, $a\in\FF^\times$;\end{center} their
dimensions are, respectively, $0, 17, 26, 26$. The orbits of
$E_6\times \FF^\times$ in $V_{27}$ are the following (lower indices
indicate dimension):
$$
{\mc O}_0 = \{0\} \;,\quad {\mc O}_{17} = X(V_{27}) \;,\quad {\mc O}_{26} = \{DET=0\}\setminus \ol{X(V_{27})} \;,\quad {\mc O}_{27}=\{DET\ne 0\} \;.
$$

{\rm (iii)} There are exactly three $E_6$ orbits in the projective
space $\PP(V_{27})$ and they are exactly the rank subsets with
respect to $\XX(V_{27})$, namely,
$$
\XX(V_{27})\;, \quad \XX_2(V_{27})=\{DET=0\}\setminus\XX(V_{27})\;, \quad \XX_3(V_{27})=\{DET\ne 0\} \; ;
$$
their dimensions are, respectively, $16, 25, 26$. The secant varieties of $\XX(V_{27})$ are exactly the closures of the $E_6$-orbits in $\PP(V_{27})$.

{\rm (iv)} The stabilizer of any vector $v\in \{DET\ne 0\}$ is isomorphic to $F_4$. The orth-complement $v^\perp\subset V_{27}$ is an irreducible $F_4$-module isomorphic to $V_{26}$, i.e. $V_{27} \cong \langle v \rangle\oplus V_{26}$ as $F_4$-modules.

{\rm (v)} The secant varieties of $\XX(V_{26})$ are obtained as intersections of the secant varieties of $\XX(V_{27})$ with the hyperplane $\PP(V_{26})$, i.e.
$$
\XX(V_{26}) = \PP(V_{26})\cap \XX(V_{27}) \;,\quad \sigma_2(\XX(V_{26})) = \PP(V_{26})\cap \sigma_2(\XX(V_{27})) \;,\quad \sigma_3(\XX(V_{26}))=\PP(V_{26}) \;.
$$
\end{lemma}

\begin{proof}
Since the results are known, but are a compilation of the work of
many authors, we confine ourselves to giving references (not
necessarily the original ones) for the various parts of the lemma.
Part (i) can be found in Table II in \cite{Kac-1980}. As for part
(ii), the fact that $$\{DET=a\}, a\ne 0$$ is a single $E_6$-orbit,
is proven in~\cite[Proposition 1.1]{Kac-1980}, while the enumeration
of the orbits in the nullcone $\{DET=0\}$ is given in \cite[p.~59]{Zak-Book}. Part (iii) can also be deduced from the discussion on
p.~59 of \cite{Zak-Book} or can be seen to follow directly from the
fact that $V_{27}$ is a subcominuscule representation and for such
representations the rank sets are exactly the group orbits in the
projective space, cf. \cite[\S 4]{Bucz-Lands-2011}. Parts (iv) and
(v) are also quoted from \cite[p.~59-60]{Zak-Book}.
\end{proof}

The above proposition and, specifically, parts (iv) and (v) allow us
to practically forget about the group $F_4$ and use only properties
of $V_{27}$ and a generic hyperplane inside it. We shall need to
understand the structure of $V_{27}$ with respect to a subgroup of
$E_6$ of type $D_5$. Let $H\subset E_6$ be the regular subgroup
whose root system is generated by the set of simple roots
$S:=\{\alpha_2,\alpha_3,\alpha_4,\alpha_5,\alpha_6\}$ (we use the
numbering of simple roots as in \cite[p. 292]{VO}). It turns our
that $H\cong Spin_{10}$.

\begin{lemma}
The decomposition, as an $H$-module, of the simple $E_6$-module $V_{27}$ is
$$
V_{27} \cong V_1 \oplus RSpin_{10} \oplus V_{10} \;,
$$
where $V_1$ is a one-dimensional trivial $H$-module, $RSpin_{10}$ is
the spinor $H$-module, and $V_{10}$ is the natural representation of
$SO_{10}$ (recall that $Spin_{10}$ is a cover of $SO_{10}$).
\end{lemma}

\begin{proof}
The result is obtained by a straightforward consideration of the weights of the modules involved, using the fact that $H$ is a regular subgroup of $E_6$.
\end{proof}

\begin{lemma}\label{Lp26}
The nonzero isotropic vectors in $V_{10}$ belong to $X(V_{27})$ and have rank 1 as elements of the $E_6$-module $V_{27}$. The non-isotropic vectors in $V_{10}$ belong to ${\mc O}_{26}$ and have rank 2 as elements of the $E_6$-module $V_{27}$.
\end{lemma}

\begin{proof}
Since $H$ is a regular subgroup of $E_6$, the weight spaces for
$E_6$ in $V_{27}$ are also weight spaces for $H$. Thus $V_{10}$ is a
span of some of these weight spaces. The $E_6$-weights of $V_{27}$
are
$$
\varepsilon_i\pm\varepsilon,-\varepsilon_i-\varepsilon_j \; (i\ne j).
$$
The weights appearing in $V_{10}$ are
$$
\varepsilon_i-\varepsilon, -\varepsilon_1-\varepsilon_i \; (i\ne1).
$$
The weight $-\varepsilon_6-\varepsilon$ is the lowest weight of $V_{27}$ and thus any element of the corresponding weight space belongs to $\mathcal O_{17}$. On the other hand, any element of the weight space of weight $-\varepsilon_6-\varepsilon$ is isotropic. Since all isotropic vectors of $V_{10}$ are conjugate by $SO_{10}$, all isotropic vectors of $V_{10}$ belong to $\mathcal O_{17}$.

It remains to show that all non-isotropic vectors of $V_{10}$ (they are all $SO_{10}\times\mathbb F^\times$-conjugate) belong to $\mathcal O_{26}$. To this end, we note that the weights of $V_{10}$
$$
\varepsilon_2-\varepsilon,-\varepsilon_1-\varepsilon_2
$$
are, respectively, the highest and the lowest weight of $V_{27}$ with respect to the set of simple roots of $E_6$
\begin{center}
$\Pi'=\{\varepsilon_2-\varepsilon_1,
\varepsilon_1+\varepsilon_4+\varepsilon_5-\varepsilon,
\varepsilon_6-\varepsilon_4, \varepsilon_4-\varepsilon_5,
-\varepsilon_4-\varepsilon_3-\varepsilon_6+\varepsilon,
\varepsilon_3-\varepsilon_6\}.$
\end{center}
Thus $v^{\varepsilon_2-\varepsilon}+v^{-\varepsilon_1-\varepsilon_2}
\in \mathcal O_{26}$~\cite[Ch. III, Thm 1.4]{Zak-Book}. Since all
non-isotropic vectors of $V_{10}$ are $SO_{10}\times\mathbb
F^\times$-conjugate, all non-isotropic vectors of $V_{10}$ belong to
$\mathcal O_{26}$.
\end{proof}

\begin{proof}[Proof of Lemma~\ref{Lf42t}]
According to Lemma \ref{LE_6V1}, to prove that $V_{26}$ is 2-continuous it
suffices to show that for any $x\in V_{26}\cap\mathcal O_{26}$ there
exist $x_+, x_-\in V_{26}\cap\mathcal O_{17}$ such that $x=x_++x_-$.
We fix $x\in V_{26}\cap\mathcal O_{26}$. First, we note that there
exists a Borel subalgebra $\frak b\subset E_6$ with a Cartan
subalgebra $\frak h\subset\frak b$ such that $x\in V_{10}$ (we use
the notation of Lemma~\ref{Lp26}). Then $x\in V_{10}\cap \mathcal
O_{26}$ and thus $x$ is a non-isotropic vector of $V_{10}$ by
Lemma~\ref{Lp26}. Let $x^\bot$ be the orthogonal complement to $x$
in $V_{10}$. Note that a nondegenerate symmetric bilinear form
$(\cdot, \cdot)$ of $V_{10}$ is still nondegenerate after
restriction to $x^\bot$.

Since $V_{26}$ is a 26 dimensional subspace of the 27 dimensional space $V_{27}$, we have
$$
\dim(V_{26}\cap x^\bot)\ge\dim(x^\bot)-1=9.
$$
As $9>\frac12\dim V_{10}$, the restriction of $(\cdot, \cdot)$ to $x^\bot\cap V_{26}$ is non-zero. Hence there exists $y\in x^\bot\cap V_{26}$ such that $(y,y)=\frac{-(x, x)}4$. Set
\begin{center}
$x_+=\frac x2+y,\hspace{10pt}x_-=\frac x2-y$.
\end{center}
We have
\begin{center}
$(x_+, x_+)=\frac14(x, x)+(y, y)=0=(x_-, x_-)$ and $x_++x_-=x$,
\end{center}
i.e. the vectors $x_\pm$ are isotropic and their sum is equal to $x$. Thus $x_\pm\in\mathcal O_{17}$. Since $y\in V_{26}$, we have $x_\pm\in V_{26}$. Hence $x_\pm\in X(V_{26})$.
Therefore $V_{26}$ is 2-continuous.
\end{proof}

Before proceeding with the proof of Lemma \ref{Lf43t}, we need the following auxiliary result.

\begin{lemma}\label{Lrk3}
Let $V$ be a finite-dimensional vector space. Let $Z_f$ be a hypersurface determined as the zero-locus of a non-zero homogeneous polynomial $f\in\mathbb F[V]$ and let $X\in V$ be a conical subset spanning $V$. Then, $V=Z_f+X$, i.e. for every $v\in V$ there exist $v_2\in Z_f$ and $x\in X$ such that $v=v_2+x$.
\end{lemma}

\begin{proof}
Assume on the contrary, that there exists $v\in V$ such that $v\ne v_2+x$ for any $v_2\in Z_f$ and $x\in X$. Then $f(v+tx)\ne 0$ for any $x\in X$ and any $t\in\mathbb F$. The function $f(v+tx)$ is polynomial and thus $f(v+tx)\ne0$ for any $t\in\mathbb F$ if and only if $f(v+tx)$ is a non-zero constant as a polynomial of $t$. Hence the first derivative of $f(v+tx)$ with respect to $t$ is zero for $t=0$, i.e. the value of $df$ in the direction $x$ at the point $v$ is zero. As $X$ spans $V$, $df=0$ at $v$.

We claim that $f(w)=0$ for all points $w\in V$ such that $df=0$ at
$w$ (equivalent to: all partial derivatives of $f$ vanish at $w$).
Indeed, the set of equations $df=0$ determines some subvariety
$Z_{df}$ of $V$ and it suffices to show that $f=0$ at any smooth
point of any irreducible component of $Z_{df}$. Obviously $df=0$ on
the smooth locus of any irreducible component of $Z_{df}$. Thus $f$
is constant on any irreducible component of $Z_{df}$. Since $f$ is
homogeneous, $f(0)=0$, and any irreducible component of $Z_{df}$
contains 0. Thus $f|_{V_{df}}=0$.

Compiling the previous two paragraphs, we obtain $f(v)=0$. Thus $v=v+0$, where $v\in Z_f$ and $0\in X$. This completes the proof.
\end{proof}

\begin{proof}[Proof of Lemma~\ref{Lf43t}]
The secant variety $\sigma_2(X(V_{26}))$ is the zero-locus of some
homogeneous function $DET$ of degree 3 and $X(\pi_1)$ spans
$V_{26}$. Thus, according to Lemma~\ref{Lrk3}, any vector $$x\in
V_{26}\backslash\sigma_2(X(V_{26}))$$ may be represented as
$x=x_1+x_2$, for some $x_1\in X(\pi_1)$ and
$x_2\in\sigma_2(X(V_{26}))$. Therefore, by Lemma~\ref{Lf42t}, any
vector in $V_{26}\backslash\sigma_2(X(V_{26}))$ has rank 3.
\end{proof}

\subsection{R-discontinuity of $V(\pi_2)$ for $F_4$}\label{SSf4w}

The goal of this section is to prove the following.

\begin{prop}\label{Prop Fund-2 F4}
The fundamental representation $V(\pi_2)$ of $F_4$ is 2-discontinuous.
\end{prop}
We deduce this proposition from the following three lemmas (we use the description of the corresponding roots and weights given in~\cite[p. 294--295]{VO}). In particular, the highest root of ${\mk f}_4$ coincides with the fundamental weight $\pi_4$. In the rest of this section we use our standard notation applied to the representation $(F_4, V(\pi_2))$.

\begin{lemma}\label{L14-1}The $F_4$-orbit of $[g^{-\pi_4}v^{\pi_2}]$ is open in $T\XX$, where by $g^{-\pi_4}$ we denote a nonzero root vector of $\frak f_4$ with root $-\pi_4$.\end{lemma}

\begin{lemma}\label{L14-2} If all elements of $T\XX$ have rank two or less, then the $F_4$-orbit of $[v^{\pi_2}+v^{-\pi_2+\alpha_2}]$ is open in $T\XX$.\end{lemma}

\begin{lemma}\label{L14-3} The vectors $g^{-\pi_4}v^{\pi_2}$ and $v^{\pi_2}+v^{-\pi_2+\alpha_2}$ belong to different $F_4$-orbits.\end{lemma}

We present the proofs of Lemmas~\ref{L14-1} and~\ref{L14-2} consecutively.
\begin{proof}[Proof of Lemma~\ref{L14-1}] It suffices to show that the orbit $P_{\pi_2}g^{-\varepsilon_1-\varepsilon_2}$ is open in the quotient $\frak f_4/\frak p_{\pi_2}$, where $P_{\pi_2}$ denotes the stabilizer in $F_4$ of $v^{\pi_2}$ and $\frak p_{\pi_2}$ is the Lie algebra of $P_{\pi_2}$. This statement follows from the fact that $$[\frak p_{\pi_2}, g^{-\varepsilon_1-\varepsilon_2}]+\frak p_{\pi_2}=\frak f_4.$$\end{proof}

\begin{proof}[Proof of Lemma~\ref{L14-2}] First note that, by Proposition~\ref{Pl96}, we have $\dim\sigma_2(\XX)=2\dim \XX+1$ and $\dim T\XX=2\dim\XX$. Assume that all points in $T\XX$ have rank two or less. This means that $T\XX$ is contained in $\XX_2\sqcup\XX$, which, by definition, is the image of $(X\times X)_0\times\PP^1$, where $(X\times X)_0$ is the complement of the diagonal in $X\times X$. Then the preimage of $T\XX$ has to be an $F_4$-stable divisor $D'$ of $(X\times X)_0\times\PP^1$. It follows from Lemma~\ref{Lzak?} that $D'$ has to be a product of a divisor $D\subset (X\times X)_0$ and $\PP^1$. Using the fact that $\pi_2=-w_0\pi_2$ we see that there is only one $F_4$-stable divisor on $\XX\times\XX=F_4/P_{\pi_2}\times F_4/P_{\pi_2}$. This divisor has an open $F_4$-orbit and the image of this $F_4$-orbit in $\PP(\pi_2)$ equals $F_4[v^{\pi_2}+v^{w_0s_{\alpha_2}(\pi_2)}]$, where $s_{\alpha_2}$ denotes the reflection with respect to the root $\alpha_2$. It remains notice that $s_{\alpha_2}\pi_2=\pi_
2-\alpha_2$ and that $w_0=-1$ for $F_4$.\end{proof}

\begin{proof}[Proof of Lemma~\ref{L14-3}]The proof is based on the following two facts. First, the $F_4$-module $V(\pi_1)$ has an invariant non-degenerate symmetric bilinear form $(\cdot, \cdot)$ and thus $F_4\subset\SO(V(\pi_1))$. Second, the decomposition of $\Lambda^2V(\pi_1)$ as an $F_4$-module is
$$
\Lambda^2 V(\pi_1) \cong V(\pi_2)\oplus V(\pi_4),
$$
see~\cite[Table 5 on p. 305]{VO}. This allows us to represent the elements of $V(\pi_2)$ as anti-symmetric tensors and perform calculations. Essentially, we will show that $$v^{\pi_2}+v^{-\pi_2+\alpha_2}\notin\SO(V(\pi_1))(g^{-\pi_4}v^{\pi_2}).$$

From now on, we consider $V(\pi_2)$ as a subspace of $\Lambda^2V(\pi_1)$. To any $v\in \Lambda^2V(\pi_1)$ we assign $\Supp v$ as in Subsection~\ref{SSsp}. It is clear that, if $v_1, v_2\in V(\pi_2)$ are $\SO(V(\pi_1))$-conjugate, then the spaces $\Supp v_1$ and $\Supp v_2$ must be $\SO(V(\pi_1))$-conjugate and in particular $(\cdot, \cdot)$ restricted to $\Supp v_1$ and $\Supp v_2$ must have the same rank.

We have the following $\wedge$-decompositions for the vectors of $V(\pi_2)\subset\Lambda^2V(\pi_1)$
$$
v^{\pi_2}=v^{\varepsilon_1}\wedge v^{\frac{\varepsilon_1+\varepsilon_2+\varepsilon_3+\varepsilon_4}2},\hspace{10pt} v^{-\pi_2+\varepsilon_2}=v^{-\varepsilon_1}\wedge v^{\frac{-\varepsilon_1-\varepsilon_2-\varepsilon_3+\varepsilon_4}2}\;. \label{Ehwp2}
$$
We calculate
$$
v^{\pi_2}+v^{-\pi_2+\varepsilon_2}=v^{\varepsilon_1}\wedge v^{\frac{\varepsilon_1+\varepsilon_2+\varepsilon_3+\varepsilon_4}2}+v^{-\varepsilon_1}\wedge v^{\frac{-\varepsilon_1-\varepsilon_2-\varepsilon_3+\varepsilon_4}2},
$$
$$
\Supp (v^{\varepsilon_1}\wedge v^{\frac{\varepsilon_1+\varepsilon_2+\varepsilon_3+\varepsilon_4}2}+v^{-\varepsilon_1}\wedge v^{\frac{-\varepsilon_1-\varepsilon_2-\varepsilon_3+\varepsilon_4}2})=\langle v^{\varepsilon_1}, v^{\frac{\varepsilon_1+\varepsilon_2+\varepsilon_3+\varepsilon_4}2}, v^{-\varepsilon_1}, v^{\frac{-\varepsilon_1-\varepsilon_2-\varepsilon_3+\varepsilon_4}2}\rangle
$$
and
$$
g^{-\pi_4}v^{\pi_2}=g^{-\varepsilon_1-\varepsilon_2}(v^{\varepsilon_1}\wedge v^{\frac{\varepsilon_1+\varepsilon_2+\varepsilon_3+\varepsilon_4}2})=v^{-\varepsilon_2}\wedge v^{\frac{\varepsilon_1+\varepsilon_2+\varepsilon_3+\varepsilon_4}2}+v^{\varepsilon_1}\wedge v^{\frac{-\varepsilon_1-\varepsilon_2+\varepsilon_3+\varepsilon_4}2},
$$
$$
\Supp (v^{-\varepsilon_2}\wedge v^{\frac{\varepsilon_1+\varepsilon_2+\varepsilon_3+\varepsilon_4}2}+v^{\varepsilon_1}\wedge v^{\frac{-\varepsilon_1-\varepsilon_2+\varepsilon_3+\varepsilon_4}2}) =\langle v^{-\varepsilon_2}, v^{\frac{\varepsilon_1+\varepsilon_2+\varepsilon_3+\varepsilon_4}2}, v^{\varepsilon_1}, v^{\frac{-\varepsilon_1-\varepsilon_2+\varepsilon_3+\varepsilon_4}2}\rangle
$$
This implies that the space $\Supp (g^{-\pi_4}v^{\pi_2})$ is $(\cdot, \cdot)$-isotropic, while $\Supp (v^{\pi_2}+v^{-\pi_2+\varepsilon_2})$ is not. From this we conclude that $g^{-\pi_4}v^{\pi_2}$ and $v^{\pi_2}+v^{-\pi_2+\alpha_2}$ belong to different $\SO(V(\pi_1))$-orbits and thus to different $F_4$-orbits.
\end{proof}

\subsection{R-discontinuity of $V(\pi_1)$ for $E_7$}\label{SSSE7}

This subsection is devoted to the proof of the following proposition.

\begin{prop}\label{LE7}The representation $V(\pi_1)$ of $E_7$ is 2-discontinuous.\end{prop}

The specific notation used in this section is motivated by the guest appearance of the adjoint module of $E_8$ in the proof. We set $G=E_8$ and $\ul G=E_7$. By $\frak g, \frak h, \pi_1, ...$ and so on we denote attributes of $E_8$ and by $\ul{\frak g},\ul{\frak h},\ul{\pi_1}, ...$ we denote the corresponding attributes of $\ul G=E_7$. We refer to \cite[Table~1~on~p.~293--295]{VO} for the roots and fundamental weights of $E_8$. Note that the fundamental weight $\pi_1$ of $E_8$ coincides with the highest root, i.e. $V(\pi_1)\cong {\mk e}_8$ is the adjoint representation of $E_8$. The idea of the proof of Proposition~\ref{LE7} is to identify $V(\ul{\pi_1})$ of $E_7$ with some subspace of $\frak e_8$ and then prove the following two lemmas.

\begin{lemma}\label{Le7x2}a) For any $x\in X(\ul{\pi_1})$ we have $\dim E_8x=58$.\\b)  For any $x\in X_2(\ul{\pi_1})$ we have $\dim E_8x\in\{58, 92, 114\}$.\end{lemma}
\begin{lemma}\label{Le7rk3}There exists $x\in V(\ul{\pi_1})$ such that $\dim E_8x=112$.\end{lemma}
It is known that
$\sigma_2(X(\ul{\pi_1}))=V(\ul{\pi_1})$~\cite[Ch. III, Thm. 1.4]{Zak-Book}.
Therefore $V(\ul{\pi_1})$ is r-discontinuous if and only if there exists $x\in
V(\ul{\pi_1})$ such that $$x\not\in X_2(\ul{\pi_1})\cup
X(\ul{\pi_1})\cup0.$$ According to Lemma~\ref{Le7x2} and
Lemma~\ref{Le7rk3} such elements $x\in V(\ul{\pi_1})$ exist and
hence Lemmas~\ref{Le7x2} and~\ref{Le7rk3} imply
Proposition~\ref{LE7}. We now present some explanation of Lemma~\ref{Le7x2} and
Lemma~\ref{Le7rk3} and then proceed with their proofs.

We note the amazing fact that the $V(\ul{\pi_1})$ has only finitely
many $\ul G$-orbits and we wish to say some words about it (see e.g.
\cite{Vi}). A description of the $E_7$-orbits on $\FF^{56} (\dim
V(\ul{\pi_1})=56)$ appears in~\cite{Har}. The idea
of the description used here comes from~\cite{BaCa} and is
related to the description of $\frak{sl}_2$-triples in exceptional
groups due to~\cite{Dyn}. There is a recently developed software,
which allows, in principle, to solve such problems~\cite{GVY}.

In our proof of 2-discontinuity of $V(\ul{\pi_1})$ we use the fact that
$V(\ul{\pi_1})$ is the 1-grading component of some grading of $\frak
e_8$ (representations which arise in such a way are called
{$\theta$-representations}, see~\cite{Vi},~\cite{Kac-1980}). We need
more notation related to $\theta$-representations.

For any $t\in\frak h^*$ we denote by $\frak g_t\subset\frak g$ the
corresponding weight space (we note that $\frak g_t\ne0$ if and only
if $t\in\Delta\cup0$). We identify $\frak h$ and $\frak h^*$ via the
Cartan-Killing form and thus consider fundamental weights $\pi_i$ as
elements of $\frak h^*$. We set
$$
\Delta_i:=\{\alpha\in\Delta\cup0\mid (\alpha, \pi_1)=i\}\;,\;\;\frak g_i:=\bigoplus\limits_{t\in\Delta_i}\frak g_t~\;(i\in\FF).
$$
The spaces $\{\frak g_i\}_{i\in\mathbb F}$ form a grading of $\frak g$. The space $\frak
g_0$ is a Lie algebra and it acts in a natural way on $\frak g_i$
for any $i\in\mathbb F$. By definition, a {\it
$\theta$-representation} is the representation of $\frak g_0$ on $\frak
g_1$.

We have
$$
\frak g_i=0 \;\; {\rm if}\;\; i\not\in\{-2, -1, 0, 1,2\}\;,\;\;\frak g_0\cong\frak e_7\oplus\FF,
$$
$$
\dim\frak g_2=\dim\frak g_{-2}=1\;,\;\;\dim\frak g_1=\frak g_{-1}=56\;,\;\;\dim\frak g_0=134.
$$
We identify $\ul{\frak g}={\mk e}_7$ with $[\frak g_0, \frak g_0]$. As $\frak e_7$-modules both $\frak g_1$ and $\frak g_{-1}$ are isomorphic to $V(\ul{\pi_1})$. Further, we identify $V(\ul{\pi_1})$ with $\frak g_1$.

The following lemma plays a key role in the proof of Lemma~\ref{Le7x2}.
\begin{lemma}\label{Le73}Let $\alpha_1, \alpha_2$ be roots of $E_8$ such that $\alpha_1\ne-\alpha_2$. Then $v^{\alpha_1}+v^{\alpha_2}$ is a nilpotent element and $\dim E_8(v^{\alpha_1}+v^{\alpha_2})\in\{58, 92, 114\}$.\end{lemma}
\begin{proof} If $\alpha_1=\alpha_2$, then $v^{\alpha_1}+v^{\alpha_2}$ is conjugate to $v^{\alpha_1}$. The nilpotent element $v^{\alpha_1}$ is a generic nilpotent element of the corresponding Levi subalgebra with semisimple part isomorphic to $A_1$. Therefore $\dim E_8(v^{\alpha_1}+v^{\alpha_2})=58$.

If $\alpha_1\ne\pm\alpha_2$, the vector
$v^{\alpha_1}+v^{\alpha_2}$ is a nilpotent element of the Lie
algebra $\frak l_{\alpha_1, \alpha_2}$ corresponding to the root
system generated by $\alpha_1, \alpha_2$. We have three
possibilities: $(\alpha_1, \alpha_2)=1$, $(\alpha_1, \alpha_2)=0$,
$(\alpha_1, \alpha_2)=-1$. In the first and third cases, we have
$\frak l_{\alpha_1, \alpha_2}\cong \frak{sl}_3=A_2$. In the second
case, we have $\frak l_{\alpha_1,
\alpha_2}\cong\frak{sl}_2\oplus\frak{sl}_2=2A_1$. For any of these
Lie algebras of rank 2 it is easy to check that:

1) if $(\alpha_1, \alpha_2)=1$, then $v^{\alpha_1}+v^{\alpha_2}$ is
conjugate in $\frak l_{\alpha_1, \alpha_2}$ to $v^{\alpha_1}$ (and
therefore to $v^{\alpha_2})$, and thus is a distinguished nilpotent
element for some root subalgebra $A_1$,

2) if $(\alpha_1, \alpha_2)=0$, then $v^{\alpha_1}+v^{\alpha_2}$ is
a distinguished nilpotent element of $\frak l_{\alpha_1,
\alpha_2}\cong 2A_1$,

3) if $(\alpha_1, \alpha_2)=-1$, then $v^{\alpha_1}+v^{\alpha_2}$ is
a distinguished nilpotent element of $\frak l_{\alpha_1,
\alpha_2}\cong A_2$.

Hence $\dim E_8(v^{\alpha_1}+v^{\alpha_2})=58, 92, 114$,
respectively, for cases 1, 2, 3, see~\cite[8.4, Table: nilpotent
elements for $E_8$]{CM}.\end{proof}

\begin{proof}[Proof of Lemma~\ref{Le7x2}]For any weight
$\alpha\in\Delta_1$ and any $v^\alpha\in\frak g_\alpha$ we
have\begin{center}$v^\alpha\in$V$(\ul{\pi_1})$\end{center} and
$v^\alpha$ is a highest weight vector with respect to some choice of
Borel subalgebra of $\ul{\frak g}$,
i.e.\begin{center}$v^\alpha\in X(\ul{\pi_1})$.\end{center} On the
other hand $v^\alpha\in X(\pi_1)$ and thus\begin{center}$\dim
E_8v^\alpha=\dim X(\pi_1)=58$.\end{center}This completes part a).

We proceed to part b). By Lemma~\ref{Lzak?}, any element of
$X_2(\ul{\pi_1})$ is $\ul G$-conjugate to the sum of two weight
vectors. In our case this means that any $x\in X_2(\ul{\pi_1})$ is
$\ul G$-conjugate to $$v^{\alpha_1}+v^{\alpha_2}$$for some
$\alpha_1, \alpha_2\in\Delta_1$. From this statement and
Lemma~\ref{Le73} part b) of Lemma~\ref{Le7x2} follows
immediately.\end{proof}

\begin{proof}[Proof of Lemma~\ref{Le7rk3}] We shall construct an element $x$ with $\dim E_8x=112$. First note that all roots of $E_8$ are conjugate and that the Dynkin diagram of $E_8$ has a unique subdiagram of type $D_4$. Hence there exists roots $\alpha_1,\alpha_2, \alpha_3$ such that the quadruple $$(-\pi_1, \alpha_1,\alpha_2, \alpha_3)$$ is a system of simple roots of Dynkin type $D_4$, i.e.

1) $(-\pi_1, \alpha_i)=-1$ for $i=1, 2, 3$,

2) $(\alpha_i, \alpha_j)=0$ for $i,j\in\{1, 2, 3\}$, $i\ne j$.

Condition 1) means that $\alpha_1, \alpha_2, \alpha_3\in\Delta_1$.
The element $v^{\alpha_1}+v^{\alpha_2}+v^{\alpha_3}$ is a
distinguished element of $3A_1$, where by $3A_1$ we denote the
subgroup of $G$ corresponding to the root subsystem
$$\cup_i\{-\alpha_i, \alpha_i\}\subset\Delta.$$ Therefore $\dim
E_8(v^{\alpha_1}+v^{\alpha_2}+v^{\alpha_3})=112$, see~\cite[8.4,
Table: nilpotent elements for~$E_8$]{CM}. Hence, for
$x=(v^{\alpha_1}+v^{\alpha_2}+v^{\alpha_3})\in V(\ul{\pi_1})$, we
have $\dim E_8x=112$. This completes the proof.\end{proof}

\subsection*{Acknowledgements}
The first author wishes to thank Ruhr-Universit\"at Bochum for a
two-month stay in a very friendly environment. This project was
started at this university with the support of the grant SFB/TR12.
We would like to thank A.~Elashvili, P.~Heinzner, A.~Huckleberry, K.~Ranestad,
A.~Sawicki, D.~Timashev for stimulating discussions. We also thank the anonymous reviewer for helpful suggestions.
\bibliographystyle{plain}

\end{document}